\documentclass{article}
\usepackage[english]{babel}
\usepackage{amsmath,amssymb, amsthm}
\usepackage{graphicx}

\newcommand{\infixand}{\text{ and }}
\newcommand{\mathd}{\mathrm{d}}
\newcommand{\tmname}[1]{\textsc{#1}}
\newcommand{\tmop}[1]{\ensuremath{\operatorname{#1}}}
\newtheorem{theorem}{Theorem}
\newtheorem{lemma}{Lemma}
\newtheorem{corollary}{Corollary}[theorem]
\newtheorem{proposition}{Proposition}
\newtheorem{conjecture}{Conjecture}
\newtheorem{question}{Question}
\newtheorem{example}{Example}
\newtheorem{remark}{Remark}

\def\H{{\bf H}}
\def\S{{\bf S}}

\def\HH{{\mathcal H}}
\def\R{{\mathbb R}}
\def\C{{\mathbb C}}
\def\Z{{\mathbb Z}}
\def\T{{\mathbb T}}
\newcommand{\pv}{\operatorname{p.v.}}

\begin{document}

	\title{Invariant sets of the double Hilbert transform}
	\author{E. Abakumov \and K. Domelevo \and
		S. Petermichl\thanks{S.~Petermichl was partially supported by the Alexander von Humboldt Foundation.}
		\and A. Poltoratski\thanks{A.~Poltoratski was partially supported by NSF Grant DMS-1954085.}}
	\date{}

	\maketitle

	\begin{abstract}
		We present an example that shows that the double Hilbert transform
		$\H$ can have the characteristic function of an open set of finite measure in $\mathbb{R}
		\times \mathbb{R}$ as an approximate eigenvector for the eigenvalue $1$. The
		same holds for the dyadic model of $\H$, the double dyadic shift. We further construct
		non-diagonal sets whose indicators are fixed by $\H$ -- strips bent along
		the Boole curve $x=y-1/y$ and, more generally, along level curves of real
		Herglotz functions with purely singular measures -- and study the class of
		invariant sets which they generate. At positive density we construct an
		example built on three families of parallel lines, related to the carry
		cocycle $\{u\}+\{v\}-\{u+v\}$, and the fans -- alternate sectors of an odd number of concurrent lines -- which are the simplest members of the invariant cones, a class that we describe completely. We then study the classification of the invariant sets whose
		boundary lies on countably many lines. Its combinatorial half is settled
		completely: a locally finite arrangement of lines in which no point lies
		on exactly two lines is a family of parallel lines, a pencil of
		concurrent lines, or an affine image of the triangular lattice. The
		classification is further connected, via projective duality, with
		configurations of Sylvester--Gallai type, and, via a change of
		variables, with two-valued superpositions of traveling waves.
	\end{abstract}

	\section{Introduction}

	This note is devoted to the properties of the images of indicator functions of two-dimensional sets under the double Hilbert transform. It may be viewed as a continuation of the study of the action of the standard one-dimensional Hilbert transform on the indicator functions of subsets of the line, which has applications in complex analysis, spectral theory and functional analysis.

	Let $\mathcal{H}$ denote the Hilbert transform on the real line:
	$$\mathcal Hf(x)=\frac 1{\pi}\,\pv\int_\R \frac{f(t)}{x-t}\,dt,$$
	where $f\in L^2(\R)\cup L^1(\R)$. By $\mathcal Ff$, or $\hat f$, we denote the Fourier transform of $f$:
	$$\hat f(s)=\frac1{\sqrt{2\pi}}\int_\R e^{-ist}f(t)dt.$$
	For $f\in L^2(\R^n)\cup L^1(\R^n)$,
	$$\hat f(s)=\frac1{(2\pi)^{n/2}}\int_{\R^n} e^{-i\langle s,t\rangle}f(t)dm_n(t),$$
	where $m_n$ denotes the $n$-dimensional Lebesgue measure on $\R^n$.
	The Hilbert transform corresponds to the multiplier $m(\xi)=- i \tmop{sign} (\xi)$ on the Fourier side:
	$$\mathcal F(\HH f)(s)=m(s)\hat f(s).$$
	It is well-known that if $\chi_E$ is an indicator function of a measurable set $E\subset \R$
	of finite Lebesgue measure $|E|$, then the distribution of $\mathcal H\chi_E$ depends only on the measure of the set:
	$$|\{|\HH\chi_E|>\lambda\}|=\frac{2|E|}{\sinh (\pi\lambda)},$$
	see \cite{SW, La}. This deep fact is a reformulation of Boole's formula, which displays a similar property for the Hilbert transform, or the Cauchy integral, of a positive singular measure, see for instance \cite{Po} for further references. This property has many important implications including finite-rank perturbation theory, where the indicator function appears as the Krein spectral shift of a rank-one perturbation of a self-adjoint operator. For basic results and definitions in this area we refer the reader to \cite{K, A-D, M-P, P2, S}.

	Although the above property does not seem to have direct analogs in higher dimensions, the action of the Hilbert transform on indicator functions (indicators) of subsets of $\R^2$ has recently been studied in connection with problems in harmonic analysis and singular integrals. One of the aspects of this study is the topic of this note.

	For $f\in L^2(\R^2)\cup L^1(\R^2)$ we denote by $\HH_1$ and $\HH_2$ the Hilbert transform
	with respect to the first and second variable respectively:
	$$\HH_1f(x,y)=\frac 1{\pi}\,\pv\int_\R \frac{f(t,y)}{x-t}\,dt,$$
	$$\HH_2f(x,y)=\frac 1{\pi}\,\pv\int_\R \frac{f(x,t)}{y-t}\,dt.$$

	One of the main objects in this article is the double Hilbert transform
	$\H$:

	$$\H f=\mathcal{H}_2 \mathcal{H}_1 f=\mathcal{H}_1 \mathcal{H}_2 f.$$

	On the Fourier side $\H$ becomes an operator of multiplication
	\begin{equation}\mathcal F(\H f)=\phi\mathcal F(f)
		\label{eq1}\end{equation}
	by the function $$\phi(\xi)=1-2\chi_{Q_{1,3}}(\xi),\qquad \xi=(\xi_1,\xi_2),$$ where $Q_{1,3}$ is the union of the (open) first and third quadrants of the coordinate plane; thus $\phi=1$ on the open second and fourth quadrants, whose union we denote by $Q_{2,4}$, and $\phi=-1$ on $Q_{1,3}$. Since the multiplier $\phi$ takes only two essential values, it follows that $\pm 1$ are the only eigenvalues of
	$\H: L^2(\R^2)\to L^2(\R^2).$ Our results show that although indicators of bounded sets cannot be eigenvectors of $\H$ in $L^2$, they may appear as ``quasi-eigenvectors'' in a certain asymptotic sense. Indicators of unbounded sets appear as generalized eigenvectors outside of $L^2$.

	We start our discussion of the action of $\H$ on indicator functions with the following simple statement.

	\begin{lemma}\label{L1}
		Let $V\subset \R^2$ be a bounded set. Then $\chi_V$ is not an eigenvector of $\H$.
	\end{lemma}

	\begin{remark}
		As follows from the proof in the next section, the statement can be strengthened: bounded sets can be replaced by sets whose
		projection on one of the coordinate axes is bounded.
	\end{remark}

The actions of singular operators on indicators have attracted independent interest in various settings \cite{Va, So, Ha}. 
	It would be natural to try to quantify the last statement and to look for a universal constant $c>0$ such that for any bounded set $V\subset \R^2$,
	$$\| \H \chi_V -
	\chi_V \|_{L^2}\geq c\|\chi_V \|_{L^2}.$$
	Such questions are connected with the study of iterated commutators \cite{Ho, Vo}: indicator functions of open sets -- typically sets related to those appearing in the supremum that defines the product BMO norm -- form a natural testing class for the norm of the commutator. The double Hilbert transform arises naturally in the iterated commutator itself. The existence of a universal constant $c$ in the last inequality would rule out ``quasi-eigenfunctions'' of $\H$ for the eigenvalues $\pm 1$, which are, among other things, a substantial additional obstacle to a lower estimate of the commutator norm. However, as shown by our next statement, such a constant
$c$ does not exist.

	\begin{theorem}\label{T1}
		There exist open bounded sets $V_{\varepsilon} \subset \mathbb{R}
		\times \mathbb{R}$, $\varepsilon>0$, such that
		\begin{equation}
			\frac{\| \H \chi_{V_{\varepsilon}} -
				\chi_{V_{\varepsilon}} \|_{L^2}}{\| \chi_{V_{\varepsilon}} \|_{L^2}}
			\longrightarrow 0 \ \text{ as } \ \varepsilon \rightarrow 0. \label{asymcts}
		\end{equation}
	\end{theorem}

	Let us define the Fourier projections $\mathcal{P}^1_{\pm}$ via \eqref{eq1} with
	$\phi$ equal to the indicator function of the right, respectively left, half-plane $\{\pm\xi_1>0\}$, and $\mathcal{P}^2_{\pm}$ with $\phi$ equal to the indicator function of the upper, respectively lower, half-plane $\{\pm\xi_2>0\}$.

	An obvious modification of the sets allows one to replace $\H \chi_{V_{\varepsilon}} - \chi_{V_{\varepsilon}}$ by
	$\H \chi_{V_{\varepsilon}} + \chi_{V_{\varepsilon}}$.
	Since
	$\H -\mathcal{I}
	= - 2\mathcal{P}^1_+
	\mathcal{P}^2_+ - 2\mathcal{P}^1_- \mathcal{P}^2_-,$
	our example thus shows
	that $\| \mathcal{P}^1_+ \mathcal{P}^2_+ \chi_{V_{\varepsilon}} \|_{L^2}$ and
	$\| \mathcal{P}^1_- \mathcal{P}^2_- \chi_{V_{\varepsilon}} \|_{L^2}$ are much
	smaller than $\| \chi_{V_{\varepsilon}} \|_{L^2}$, and we immediately obtain the following corollary.

	\begin{corollary}
		There is no
		absolute constant $c>0$ such that for all open sets $V$ of finite measure $\|
		\mathcal{P}^1_+ \mathcal{P}^2_+ \chi_V \|_{L^2} \geqslant c \| \chi_V
		\|_{L^2}$.
	\end{corollary}

	Such inequalities are relevant to the estimate of the iterated commutator
	from below, and it has previously been shown by {\tmname{Volberg}} that such
	a constant cannot exist for real-valued $b$ replacing $\chi_V$. His example can
	be found in \cite{Vo}.
	These examples are in contrast to the one-parameter case,
	where the symmetric Fourier transform of a real-valued function readily
	implies such a comparison.

	One might hope that such a lower estimate is true when replacing arbitrary real-valued $b$ by indicator functions of measurable sets. As was mentioned before, in one parameter
	very precise control is known for the level sets of the Hilbert transform
	acting on indicators.
	For an indicator function $\chi_V$ of a planar set $V$ to be a generalized
	eigenfunction with eigenvalue $\pm 1$ for the double Hilbert transform
	$\H$ one therefore needs $\mathcal{H}_1
	\chi_V$ to have the right distribution not only on each horizontal line but also
	on each vertical line, so that $\mathcal{H}_2\mathcal{H}_1 \chi_V$ is again an
	indicator function.

	Interestingly, such a delicate property can be achieved
	via the simple example of a diagonal strip considered in this paper. In
	addition, we then use a diagonal strip of width $\varepsilon$ that is cut off
	in its $x_2$ variable by $\pm 1$, so as to produce a bounded set $V_{\varepsilon}$
	so that the property \eqref{asymcts} holds true.

	We also investigate such an estimate for a dyadic replacement of the
	Hilbert transform, the dyadic shift $\mathcal{S}_0$ acting on the Haar system
	via $h_{I_{\pm}} \mapsto \pm h_{I \mp}$.
	Similarly to our notations above, we denote by $\mathcal{S}_1 ,\mathcal{S}_2 $ the shifts with respect to the first and second variables and put
	$$ \S f=\mathcal{S}_2 \mathcal{S}_1 f =\mathcal{S}_1 \mathcal{S}_2 f.$$

	Again, we show

	\begin{theorem}\label{T2}
		There exist bounded sets $V_n$
		of finite measure, measurable with respect to the $\sigma$-algebra generated by
		$\mathcal{D}_n\times\mathcal{D}_n$, the dyadic squares of side $2^{- n}$, such that
		\[ \frac{\| \S \chi_{V_n} - \chi_{V_n} \|_{L^2}}{\|
			\chi_{V_n} \|_{L^2}} \longrightarrow 0 \ \text{ as } \ n \rightarrow \infty .\]
	\end{theorem}

	Interestingly, although our examples are related, the width of our strip
	is $\varepsilon$ and that of our dyadic strip is $2^{- n}$, we get
	\[ \frac{\|  \H \chi_{V_{\varepsilon}} -
		\chi_{V_{\varepsilon}} \|_{L^2}}{\| \chi_{V_{\varepsilon}} \|_{L^2}}
	\lesssim \sqrt{\varepsilon | \log \varepsilon |} \]
	while we get
	\[ \frac{\| \S \chi_{V_n} - \chi_{V_n} \|_{L^2}}{\|
		\chi_{V_n} \|_{L^2}} \lesssim \sqrt{2^{- n}} , \]
        without the logarithmic term. 

	The remaining sections of the note study the exactly invariant sets. In Section \ref{sec:bent} we show that diagonal sets, i.e. unions of parallel lines of positive slope, are far from the only sets whose indicators are fixed by $\H$: any strip may be bent along the level curves of a real Herglotz function with purely singular measure. The simplest example is the hyperbolic strip
	$$S_\varepsilon=\left\{(x,y):\ \left|x-\left(y-\frac 1y\right)\right|<\varepsilon\right\},$$
	a union of translates of a two-branch hyperbola which contains no straight line and satisfies $\H\chi_{S_\varepsilon}=\chi_{S_\varepsilon}$, with all horizontal and vertical cross-sections of measure exactly $2\varepsilon$. Section \ref{sec:class} develops necessary conditions for invariance -- constancy of the cross-sectional measures and a Herglotz structure of the slice barycenters -- and gives a classification in two settings: sets with connected cross-sections, and sets whose boundary lies on countably many lines (the positive-density part of the second setting being partly conjectural; its combinatorial half is settled in Theorem \ref{Tarr}: a locally finite line arrangement without ordinary intersection points is a family of parallel lines, a pencil of concurrent lines, or an affine image of the triangular lattice). In the second setting a new phenomenon appears at positive density: an example built from three families of parallel lines via the carry cocycle $\{u\}+\{v\}-\{u+v\}$, for which $2\chi_S-1$ is fixed by $\H$, and the fans, alternate sectors of an odd number of concurrent lines, whose piecewise constant profiles put them in a different regime from the carry set. The fans are the simplest of the invariant cones, which we describe completely in Theorem \ref{Tcone}: the invariant cones with a given vertex are parametrized by arbitrary measurable sets of directions of positive slope. Section \ref{sec:SG} translates the line-boundary classification, via projective duality, into a configuration problem of Sylvester--Gallai type and identifies the recent example of B\'ar\'any, Du, Schwarz, Yuan and Zamfirescu \cite{BDSYZ} as precisely the projective dual of an affine copy of the triangular lattice, the arrangement behind our positive-density example; the dual form of Theorem \ref{Tarr} shows that, essentially, there is no other example of its kind with a locally finite dual arrangement. Finally, Section \ref{sec:waves} recasts the line-supported case as a problem about waves: a sum $\sum_{d=1}^nf_d(x-a_dt)$ of traveling profiles with distinct speeds can take only the values $0$ and $1$ at all times when $n=1$ or $n=3$, the triple case being an exact three-wave resonance, and, with step profiles, for every odd $n$, the fronts then passing through a common event; this formulation connects the classification with tomography, with periodic decompositions of integer-valued functions, and with the theory of Fourier quasicrystals.

	\section{Continuous case}

	Before giving a simple proof of Lemma \ref{L1}, let us point out that according to one of the basic statements of the uncertainty principle, a non-trivial function with a bounded support in $\R^n$
	cannot have its Fourier transform vanish on a nonempty open set. It follows from \eqref{eq1} that the Fourier transform of an eigenvector of $\H$ with eigenvalue $1$ is supported in $\overline{Q_{2,4}}$, and that of an eigenvector with eigenvalue $-1$ in $\overline{Q_{1,3}}$; in either case the Fourier transform vanishes on the union of two open quadrants. Hence, such functions cannot have bounded support, let alone be characteristic functions of bounded sets. Below we supply an elementary proof without the use of the uncertainty principle.

	\begin{proof}[Proof of Lemma \ref{L1}]
		Let $V\subset [-C,C]\times [-C,C]$ and let $l_d=\{(x,y)|y=d\}$ be a horizontal line.
		Notice that for every $d$ such that $|l_d\cap V|>0$ we have
		$$\int_{|x|>C}|\mathcal{H}_1\chi_V(x,d)|^2dx>0.$$
		From the unitarity of the Hilbert transform it follows that
		\begin{align*}
		\iint_{|x|<C}|\mathcal{H}_1\chi_V(x,y)|^2\,dxdy&=\iint_{\R^2}|\mathcal{H}_1\chi_V|^2-
		\iint_{|x|>C}|\mathcal{H}_1\chi_V|^2\\
		&=\|\chi_V\|^2_{L^2(\R^2)}-\iint_{|x|>C}|\mathcal{H}_1\chi_V|^2<\|\chi_V\|^2_{L^2(\R^2)}.
		\end{align*}
		Furthermore, by the unitarity of $\HH_2$ on each vertical line,
		\begin{align*}
		\|\chi_V \H\chi_V\|^2_{L^2(\R^2)}&\leq \|\chi_{\{|x|<C\}} \H\chi_V\|^2_{L^2(\R^2)}\\
		&=\iint_{|x|<C}|\mathcal{H}_1\chi_V(x,y)|^2\,dxdy<\|\chi_V\|^2_{L^2(\R^2)},
		\end{align*}
		which implies that $\H\chi_V\neq \pm \chi_V$.
	\end{proof}

	Consider the strip $V_{\varepsilon, \infty} = \{ (x_1, x_2) : | x_1 - x_2 | <
	\varepsilon \}$. Let us first show that (formally) the strip indicator
	\[ \chi_{V_{\varepsilon, \infty}} (x_1, x_2) = \chi^{\varepsilon} (x_1 - x_2)
	= \chi^{\varepsilon} (x_2 - x_1) = \chi^{\varepsilon}_{x_2} (x_1) =
	\chi_{x_1}^{\varepsilon} (x_2) \]
	is an eigenvector of $\H$. Here we use the notation
	$\chi^{\varepsilon} = \chi_{(- \varepsilon, \varepsilon)}$ as well as
	$f_{\tau}$ for the translation of the function $f$, so $f_{\tau} (x) = f (x -
	\tau)$. Recall that $\mathcal{H}$ is translation invariant and maps even
	functions to odd functions.
	\begin{align*}
		\H \chi_{V_{\varepsilon, \infty}} (x_1, x_2)
		& = \mathcal{H}_1\mathcal{H}_2\, \chi_{x_1}^{\varepsilon} (x_2) \\[.1em]
		& =\mathcal{H}_1 (\mathcal{H} \chi^{\varepsilon})_{x_1} (x_2) \\[.1em]
		& =\mathcal{H}_1 (\mathcal{H}\chi^{\varepsilon}) (x_2 - x_1) \\[.1em]
		&= -\mathcal{H}_1 (\mathcal{H}\chi^{\varepsilon}) (x_1 - x_2)\\[.1em]
		& =  -\mathcal{H}_1 (\mathcal{H} \chi^{\varepsilon})_{x_2} (x_1) \\[.1em]
		&= -(\mathcal{H}\mathcal{H} \chi^{\varepsilon})_{x_2} (x_1) \\[.1em]
		& = \chi^{\varepsilon}(x_1 - x_2)\\[.1em]
		& = \chi_{V_{\varepsilon,\infty}} (x_1, x_2) .
	\end{align*}
	To make this into a valid $L^2$ example, we will cut off the strip. Consider
	the set
	\[ V_{\varepsilon} = \{ (x_1, x_2) : | x_1 - x_2 | < \varepsilon, | x_2 | < 1
	\} = V_{\varepsilon, \infty} \cap \{ (x_1, x_2) : | x_2 | < 1 \} . \]
	We compute the Fourier transform of $\chi_{V_{\varepsilon}}$, writing $\chi$
	for $\chi^1$
	\begin{align*}
		\mathcal{F}_2 \mathcal{F}_1 \chi^{\varepsilon} (x_1 - x_2) \chi (x_2)
		& =
		\mathcal{F}_2 \varepsilon \hat{\chi} (\varepsilon \xi_1) e^{- i x_2 \xi_1}
		\chi (x_2) \\[.7em]
		&= \varepsilon \hat{\chi} (\varepsilon \xi_1) \hat{\chi} (\xi_1 +
		\xi_2) \\[.9em]
		&= \frac{2}{\pi} \varepsilon \frac{\sin (\varepsilon
			\xi_1)}{\varepsilon \xi_1} \frac{\sin (\xi_1 + \xi_2)}{\xi_1 + \xi_2} .
	\end{align*}
	Recall again that $\mathcal{H}_j= - i (\mathcal P^j_+ - \mathcal P^j_-)$, so
	\begin{align*}
		\H \chi_{V_{\varepsilon}}
		& =  (- i)^2 (\mathcal P^1_+ -\mathcal P^1_-) (\mathcal P^2_+ - \mathcal P^2_-) \chi_{V_{\varepsilon}}\\
		& = - \mathcal P^1_+ \mathcal P^2_+
		\chi_{V_{\varepsilon}} - \mathcal P^1_- \mathcal P^2_- \chi_{V_{\varepsilon}} + \mathcal P^1_+ \mathcal P^2_-
		\chi_{V_{\varepsilon}} + \mathcal P^1_- \mathcal P^2_+ \chi_{V_{\varepsilon}} \\[1em]
		\chi_{V_{\varepsilon}}
		& =  (\mathcal P^1_+ + \mathcal P^1_-) (\mathcal P^2_+ + \mathcal P^2_-)\chi_{V_{\varepsilon}}\\
		& = \mathcal P^1_+ \mathcal P^2_+ \chi_{V_{\varepsilon}} + \mathcal P^1_- \mathcal P^2_-
		\chi_{V_{\varepsilon}} + \mathcal P^1_+ \mathcal P^2_- \chi_{V_{\varepsilon}} + \mathcal P^1_- \mathcal P^2_+
		\chi_{V_{\varepsilon}} .
	\end{align*}
	Therefore $\H \chi_{V_{\varepsilon}} -
	\chi_{V_{\varepsilon}} = - 2 \mathcal P^1_+ \mathcal P^2_+ \chi_{V_{\varepsilon}} - 2 \mathcal P^1_-
	\mathcal P^2_- \chi_{V_{\varepsilon}}$, and the two terms are orthogonal. We calculate
	\begin{align*}
		\int^0_{- \infty} \int^0_{- \infty} \left| \widehat{\chi_{V_{\varepsilon}}}
		(\xi_1, \xi_2) \right|^2 \mathd \xi_1 \mathd \xi_2
		&= \int^{\infty}_0
		\int^{\infty}_0 \left| \widehat{\chi_{V_{\varepsilon}}} (\xi_1, \xi_2)
		\right|^2 \mathd \xi_1 \mathd \xi_2 \\
		& =  \frac{4}{\pi^2} \int^{\infty}_0
		\int^{\infty}_t \frac{\sin^2 (\varepsilon t)}{t^2} \frac{\sin^2 (p)}{p^2}
		\mathd p \mathd t.
	\end{align*}
	We estimate this integral from above. If $t \geqslant 1$ then we will use
	\[ \int^{\infty}_t \frac{\sin^2 (p)}{p^2} \mathd p \leqslant \int^{\infty}_t
	\frac{1}{p^2} \mathd p = \frac{1}{t} . \]
	If $0 < t < 1$ then
	\begin{align*}
		\int^{\infty}_t \frac{\sin^2 (p)}{p^2} \mathd p
		&= \int^1_t \frac{\sin^2(p)}{p^2} \mathd p + \int^{\infty}_1 \frac{\sin^2 (p)}{p^2} \mathd p \\
		&\leqslant\int^1_t \frac{\sin^2 (p)}{p^2} \mathd p + 1 \leqslant 2 .
	\end{align*}
	We split the integral into three regions:
	\begin{align*}
		&\int^1_0 \int^{\infty}_t \frac{\sin^2 (\varepsilon t)}{t^2} \frac{\sin^2
			(p)}{p^2} \mathd p \mathd t \leqslant 2 \int^1_0 \frac{\sin^2 (\varepsilon
			t)}{t^2} \mathd t
		= 2 \varepsilon^2 \int^1_0 \frac{\sin^2 (\varepsilon
			t)}{(\varepsilon t)^2} \mathd t \leqslant 2 \varepsilon^2 . \\
		& \int^{1 / \varepsilon}_1 \int^{\infty}_t \frac{\sin^2 (\varepsilon t)}{t^2}
		\frac{\sin^2 (p)}{p^2} \mathd p \mathd t
		\leqslant \varepsilon^2 \int^{1 /
			\varepsilon}_1 \frac{\sin^2 (\varepsilon t)}{(\varepsilon t)^2} \frac{1}{t}
		\mathd t \leqslant \varepsilon^2 \int^{1 / \varepsilon}_1 \frac{1}{t}
		\mathd t = \varepsilon^2 | \log \varepsilon | . \\
		& \int^{\infty}_{1 / \varepsilon} \int^{\infty}_t \frac{\sin^2 (\varepsilon
			t)}{t^2} \frac{\sin^2 (p)}{p^2} \mathd p \mathd t \leqslant
		\int^{\infty}_{1 / \varepsilon} \frac{\sin^2 (\varepsilon t)}{t^2}
		\frac{1}{t} \mathd t \leqslant \int^{\infty}_{1 / \varepsilon}
		\frac{1}{t^3} \mathd t
		= \frac{1}{2} \varepsilon^2 .
	\end{align*}
	Together we get that
	\[ \int \int | \H \chi_{V_{\varepsilon}} -
	\chi_{V_{\varepsilon}} |^2 \lesssim \varepsilon^2 | \log \varepsilon | \]
	for sufficiently small $\varepsilon$. At the same time
	\[ \int \int | \chi_{V_{\varepsilon}} (x_1, x_2) |^2 = | V_{\varepsilon} | =
	4\varepsilon , \]
	which proves Theorem \ref{T1}.

	\section{Dyadic case}

	There exists a similar example in the dyadic world, with a very different
	calculation. Recall the dyadic grid $\mathcal{D}$:
	\[ \mathcal{D}_n = \{ 2^{- n} ([0, 1) + k) : k \in \mathbb{Z} \} \infixand
	\mathcal{D}= \cup_{n \in \mathbb{Z}} \mathcal{D}_n . \]
	For $I \in \mathcal{D}$ define the $L^{\infty}$ normalized Haar function
	$h^{\infty}_I = \chi_{I_+} - \chi_{I_-}$, where $I_+$ is the right half and
	$I_-$ is the left half of $I$. If we normalize these in $L^2$, we obtain an
	orthonormal basis consisting of $h_I = | I |^{- 1 / 2} h^{\infty}_I$ for $I
	\in \mathcal{D}$. One choice of dyadic Hilbert transform is the operator
	densely defined by
	\[ \mathcal{S}_0 : h_{I_{\pm}} \mapsto \pm h_{I_{\mp}} . \]
	The index $0$ is a reminder that this operator has mean zero in the sense of
	\cite{Pe}.

	Let us consider the dyadic strip
	\[ V_{\infty}^{\tmop{dy}} = \bigcup_{I \in \mathcal{D}_0} I \times I \]
	and its indicator function
	\[ \chi_{V_{\infty}^{\tmop{dy}}} (x_1, x_2) = \sum_{I \in \mathcal{D}_0}
	\chi_I (x_1) \chi_I (x_2) . \]
	We show that, formally, $\chi_{V_{\infty}^{\tmop{dy}}}$ is an eigenvector for
	$\S$ to the eigenvalue 1. First, observe that if $I$
	and $I'$ are dyadic siblings and $\hat{I}$ their dyadic parent, then
	\[ \chi_I (x_1) \chi_I (x_2) + \chi_{I'} (x_1) \chi_{I'} (x_2) = \frac{1}{2}
	\chi_{\hat{I}} (x_1) \chi_{\hat{I}} (x_2) + \frac{1}{2}
	h^{\infty}_{\hat{I}} (x_1) h^{\infty}_{\hat{I}} (x_2) . \]
	Therefore,
	\begin{eqnarray*}
		\chi_{V_{\infty}^{\tmop{dy}}} (x_1, x_2) & = & \sum_{I \in \mathcal{D}_{-
				1}} [\chi_{I_-} (x_1) \chi_{I_-} (x_2) + \chi_{I_+} (x_1) \chi_{I_+}
		(x_2)]\\
		& = & \frac{1}{2} \sum_{I \in \mathcal{D}_{- 1}} \chi_I (x_1) \chi_I (x_2)
		+ \frac{1}{2} \sum_{I \in \mathcal{D}_{- 1}} h^{\infty}_I (x_1) h^{\infty}_I
		(x_2) .
	\end{eqnarray*}
	Iteration leads in the limit to
	\[ \chi_{V_{\infty}^{\tmop{dy}}} (x_1, x_2) = \sum_{k = 1}^{\infty}
	\frac{1}{2^k} \sum_{I \in \mathcal{D}_{- k}} h^{\infty}_I (x_1)
	h^{\infty}_I (x_2) . \]
	Now we observe that for every dyadic interval $I$ with sibling $I'$ there holds
	\[ \S\, h^{\infty}_I (x_1) h^{\infty}_I (x_2) =
	h^{\infty}_{I'} (x_1) h^{\infty}_{I'} (x_2), \]
	since the two signs produced by $\mathcal S_0$ in the two variables cancel. Thus the individual products are not fixed by $\S$; however, $\S$ swaps the contributions of the two siblings, and since every generation $\mathcal D_{-k}$ is a disjoint union of sibling pairs, each generation sum $\sum_{I\in\mathcal D_{-k}} h^\infty_I(x_1)h^\infty_I(x_2)$ is invariant. This implies $\S \chi_{V_{\infty}^{\tmop{dy}}} =
	\chi_{V_{\infty}^{\tmop{dy}}}$.

	Now we cut off the strip as follows. Put the diagonal chain of squares of
	side length $2^{- n}$ into the unit square $I_0 \times I_0$, $I_0=[0,1)$. Call this set
	$V^{\tmop{dy}}_n = \bigcup_{I \in \mathcal{D}_n (I_0)} I \times I$, where $\mathcal D_m(I_0)$ denotes the family of dyadic subintervals of $I_0$ of length $2^{-m}$. Notice that
	\[ \| \chi_{V^{\tmop{dy}}_n} \|^2_{L^2} = 2^n (2^{- n})^2 = 2^{- n} . \]

	We use the same computation as above to obtain
	\[ \chi_{V^{\tmop{dy}}_n} = \frac{1}{2^n} \chi_{I_0} (x_1) \chi_{I_0} (x_2) +
	\sum_{k = 1}^n \frac{1}{2^k} \sum_{I \in \mathcal{D}_{n - k}(I_0)} h^{\infty}_I
	(x_1) h^{\infty}_I (x_2) . \]
	For $1\leq k\leq n-1$ the generation $\mathcal D_{n-k}(I_0)$ is a disjoint union of sibling pairs contained in $I_0$, so the corresponding Haar sums are fixed by $\S$. The top term $k=n$ consists of the single product $2^{-n}h^\infty_{I_0}(x_1)h^\infty_{I_0}(x_2)$, whose sibling partner is absent from the sum: since $I_0=[0,2)_-$, we have $\mathcal S_0 h^\infty_{I_0}=-h^\infty_{[1,2)}$ and hence $\S \left(h^\infty_{I_0}\otimes h^\infty_{I_0}\right)=h^\infty_{[1,2)}\otimes h^\infty_{[1,2)}$.
	Thus
	\[ \S (\chi_{V^{\tmop{dy}}_n}) -
	\chi_{V^{\tmop{dy}}_n} = \frac{1}{2^n}\left[ \S
	(\chi_{I_0 \times I_0}) - \chi_{I_0 \times I_0}\right] + \frac{1}{2^n}\left[h^\infty_{[1,2)}\otimes h^\infty_{[1,2)} - h^\infty_{I_0}\otimes h^\infty_{I_0}\right] . \]
	Expanding $\chi_{I_0} = -\sum_{k \geq 1} 2^{-k/2}\, h_{[0,2^k)}$ and using $\mathcal S_0 h_{[0,2^k)} = -h_{[2^k,2^{k+1})}$, we see that $\S(\chi_{I_0\times I_0})$ is supported in $[2,\infty)\times[2,\infty)$. The four terms on the right therefore split into three groups with pairwise disjoint supports, contained in $[2,\infty)^2$, $[1,2)^2$ and $[0,1)^2$ respectively, and on the unit square the two functions $\chi_{I_0\times I_0}$ and $h^\infty_{I_0}\otimes h^\infty_{I_0}$ are orthogonal. Using that
	$\S$ is an isometry, we obtain
	\[ \| \S (\chi_{V^{\tmop{dy}}_n}) -
	\chi_{V^{\tmop{dy}}_n} \|^2_{L^2} = \frac{1+1+2}{2^{2 n}}=\frac{4}{2^{2n}} , \]
	so that the ratio in Theorem \ref{T2} equals $2\sqrt{2^{-n}}$. This proves Theorem \ref{T2}.

	\section{Bent strips: non-diagonal invariant sets}\label{sec:bent}

	We call a measurable set $D\subset\R^2$ a \emph{diagonal set} if it is a union of parallel lines of positive slope, that is, $D=\{(x,y):\ \alpha y-\beta x\in E\}$ for some $\alpha,\beta>0$ and a measurable $E\subset\R$. The computation of Section 2 shows, more generally, that $\H\chi_D=\chi_D$, in the natural slice-wise sense, for every diagonal set whose cross-sections have finite measure. In this section we show that diagonal sets are far from the only invariant sets: the lines may be bent along the level curves of Herglotz functions with purely singular measures. In this section, invariance $\H\chi_S=\chi_S$ is understood slice-wise: for a.e.\ horizontal line the transform $\HH_1$ is applied to the ($L^1\cap L^2$) slice of $\chi_S$, for a.e.\ vertical line the transform $\HH_2$ is applied to the resulting function, and the output coincides with $\chi_S$ a.e.; in all examples of this section every step is a classical one-dimensional transform. (Sets with cross-sections of infinite measure require a different interpretation, discussed in Section \ref{sec:class}.)

	We write $\C_+=\{\Im w>0\}$, $\C_-=\{\Im w<0\}$. Recall that a (real) Herglotz function is an analytic map $\phi:\C_+\to\overline{\C_+}$; it admits the representation
	\begin{equation}\label{herglotz}
		\phi(w)=\beta w+\gamma+\int_\R\left(\frac 1{s-w}-\frac s{1+s^2}\right)d\mu(s),
		\qquad \beta\geq 0,\ \gamma\in\R,
	\end{equation}
	with a positive measure $\mu$ satisfying $\int (1+s^2)^{-1}d\mu<\infty$. The same formula defines $\phi$ on $\C_-$, with $\phi(\bar w)=\overline{\phi(w)}$, and maps $\C_-$ into $\C_-$; if $\mu$ is purely singular then the boundary values of $\phi$ on $\R$ are real a.e. We call such $\phi$ (with $\beta>0$) functions of the \emph{singular Herglotz class}. The model example, with $\beta=1$ and $\mu=\delta_0$, is Boole's function
	$$\phi(y)=y-\frac 1y.$$

	\begin{lemma}[Boole; commutation]\label{Lboole}
		Let $\phi(y)=y-1/y$. Then:

		(i) every $s\in\R$ has exactly two preimages $y_1\in(0,\infty)$, $y_2\in(-\infty,0)$ under $\phi$, they satisfy $y_1y_2=-1$ and
		$\frac 1{\phi'(y_1)}+\frac 1{\phi'(y_2)}=1$. Consequently $|\phi^{-1}(A)|=|A|$ for every Borel $A\subset\R$ and $V_\phi f=f\circ\phi$ is an isometry of $L^2(\R)$ \cite{Bo}.

		(ii) $\HH(f\circ\phi)=(\HH f)\circ\phi$ for every $f\in L^2(\R)$.
	\end{lemma}

	\begin{proof}
		(i) The equation $y-1/y=s$ is $y^2-sy-1=0$, with roots of product $-1$. Since $\phi'(y)=1+y^{-2}=(1+y^2)/y^2$, we get $1/\phi'(y)=y^2/(1+y^2)$ and, with $y_2=-1/y_1$,
		$$\frac{y_1^2}{1+y_1^2}+\frac{1}{1+y_1^2}=1.$$
		The change of variables formula then gives $|\phi^{-1}(A)|=|A|$.

		(ii) It suffices to verify the identity on the resolvents $f=(\cdot-z)^{-1}$, $\Im z\neq 0$, whose linear span is dense in $L^2$, since both sides are bounded operators by (i). Let $\Im z>0$ and let $t_1,t_2$ be the roots of $\phi(t)=z$. From $t_1t_2=-1$ we get $\Im t_2=\Im t_1/|t_1|^2$, so the roots lie in the same open half-plane; since $\phi(\C_-)\subset\C_-$, they lie in $\C_+$. Partial fractions give
		$$\frac 1{\phi(t)-z}=\frac{A}{t-t_1}+\frac{B}{t-t_2},\qquad A+B=1.$$
		For $\Im t_j>0$ the function $y\mapsto (y-t_j)^{-1}$ is analytic in $\C_-$, so its spectrum lies in $(-\infty,0]$ and $\HH$ acts on it as multiplication by $i$. Hence both $\HH(f\circ\phi)$ and $(\HH f)\circ\phi$ are equal to $i(\phi(y)-z)^{-1}$. The case $\Im z<0$ follows by conjugation.
	\end{proof}

	\begin{remark}
		The same proof applies verbatim to any $\phi$ of the singular Herglotz class \eqref{herglotz} with finitely many atoms, and by approximation to the whole class: all solutions of $\phi(t)=z$, $\Im z>0$, lie in $\C_+$, the partial fraction coefficients sum to $1/\beta$, and $\sqrt\beta\, V_\phi$ is an isometry of $L^2(\R)$ commuting with $\HH$. The identity $|\phi^{-1}(A)|=|A|/\beta$ for the whole class is a classical extension of Boole's lemma; see \cite{Bo, Le} and, in the language of rank-one perturbations, the spectral averaging formula in \cite{S}.
	\end{remark}

	\begin{theorem}[bent strips]\label{Tbent}
		Let $E\subset\R$ be measurable with $0<|E|<\infty$, let $\phi$ belong to the singular Herglotz class with linear coefficient $\beta>0$, and put
		$$S=\{(x,y):\ x-\phi(y)\in E\}.$$
		Then every horizontal cross-section of $S$ has measure $|E|$, every vertical cross-section has measure $|E|/\beta$, and
		$$\H\chi_S=\chi_S.$$
	\end{theorem}

	\begin{proof}
		The horizontal slice at height $y$ is $\phi(y)+E$; the vertical slice at $x$ is $\phi^{-1}(x-E)$, of measure $|E|/\beta$ by the remark above. For the invariance, fix $y$; by translation invariance
		$$\HH_1\chi_S(x,y)=(\HH\chi_E)(x-\phi(y)).$$
		Fix $x$ and let $g=\HH\chi_E\in L^2$. The function $y\mapsto g(x-\phi(y))$ lies in $L^2(dy)$, since $\int|g(x-\phi(y))|^2dy=\beta^{-1}\int|g(x-s)|^2ds<\infty$. Writing $h(s)=g(x-s)$ we have $g(x-\phi(\cdot))=h\circ\phi$, so by Lemma \ref{Lboole}(ii) and the anticommutation of $\HH$ with the reflection $s\mapsto x-s$,
		\begin{align*}
		\HH_2\left[g(x-\phi(\cdot))\right](y)&=(\HH h)(\phi(y))=-(\HH g)(x-\phi(y))\\
		&=-(\HH^2\chi_E)(x-\phi(y))=\chi_E(x-\phi(y)),
		\end{align*}
		using $\HH^2=-\mathcal I$.
	\end{proof}

	\begin{remark}
		On the Fourier side the mechanism is the following. The partial Fourier transform of $\chi_S$ in $x$ equals $F(t,y)=e^{-it\phi(y)}\widehat{\chi_E}(t)$ (up to a normalization constant), and for $t>0$
		$$e^{-it\phi(y)}=e^{-ity}\cdot e^{it/y}\qquad(\phi(y)=y-1/y)$$
		is the boundary value of an \emph{inner function} in $\C_-$: the exponential factor times the classical singular inner function with a point mass at the origin. Thus the spectrum of $\chi_S$ lies in $\{\xi_1>0,\ \xi_2\leq -\beta\xi_1\}$ together with its reflection, inside $\overline{Q_{2,4}}$, where the symbol of $\H$ equals $1$. Bending a strip is multiplication of the vertical wave by a singular inner factor.
	\end{remark}

	\begin{proposition}\label{Pnoline}
		Let $E$ be a bounded interval and $\phi(y)=y-1/y$. Then $S=\{x-\phi(y)\in E\}$ is an open set which contains no straight line; moreover, $S$ does not coincide up to a null set with any diagonal set.
	\end{proposition}

	\begin{proof}
		Along a nonvertical line $y=mx+c$ the quantity $x-\phi(y)=x-y+1/y$ equals $(1-m)x-c+ (mx+c)^{-1}$, which is unbounded when $m\neq 1$ and equals $-c+(x+c)^{-1}$, unbounded as $x\to -c$, when $m=1$; along a vertical line it is $c-y+1/y$, unbounded as $y\to\pm\infty$. In each case the line leaves the bounded window $E$. If $S$ coincided a.e.\ with a diagonal set $D$, then for a.e.\ $y$ the interval $\phi(y)+E$ would coincide up to a null set with an interval whose endpoint is an affine function of $y$, forcing $\phi$ to be affine, a contradiction.
	\end{proof}

	Figure \ref{fig:boole} shows $S_\varepsilon=\{|x-(y-1/y)|<\varepsilon\}$, that is, the set of Theorem \ref{Tbent} with $E=(-\varepsilon,\varepsilon)$.

	\begin{figure}[htb]
		\centering
		\includegraphics[width=.72\textwidth]{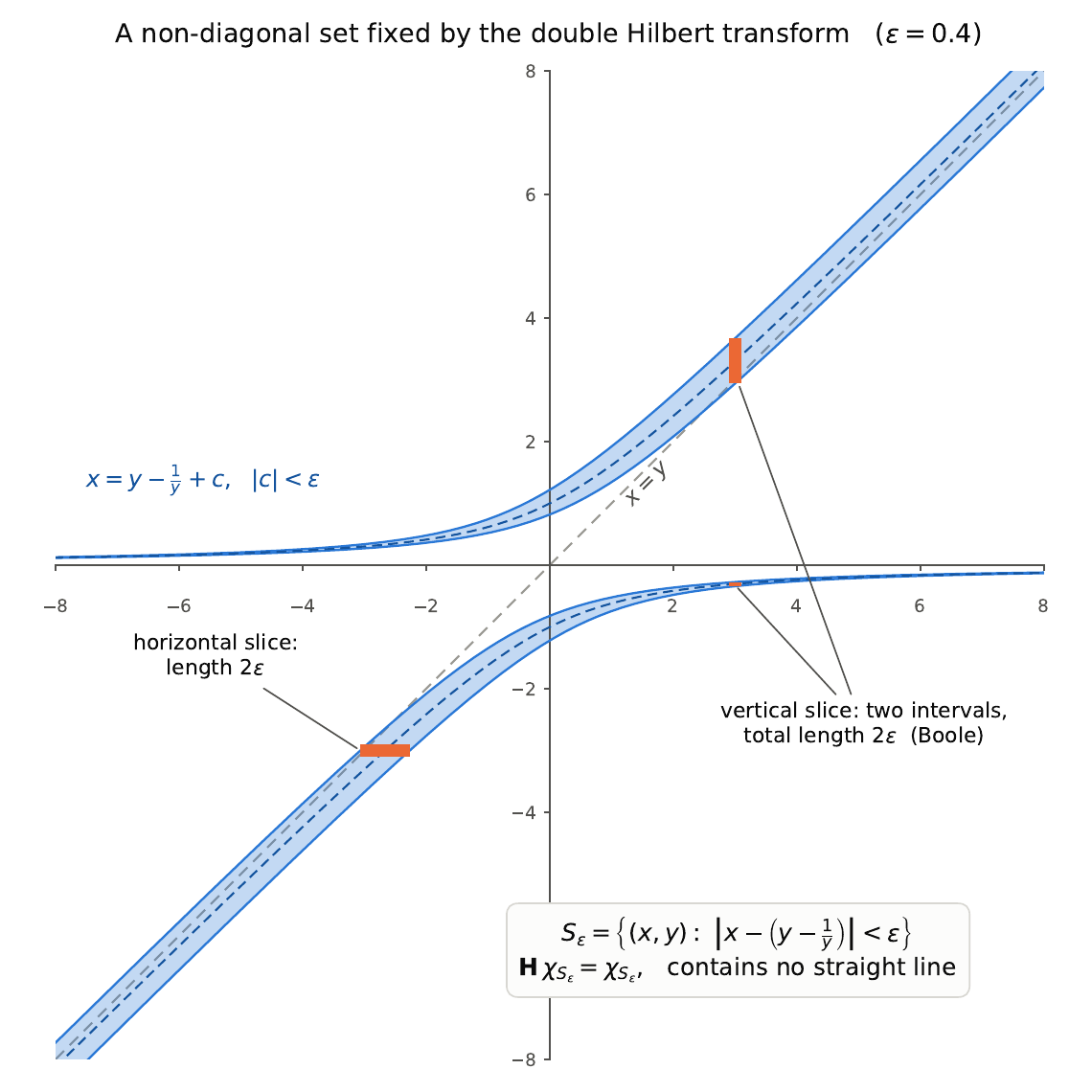}
		\caption{The bent strip $S_\varepsilon=\{|x-(y-1/y)|<\varepsilon\}$, $\varepsilon=0.4$: a union of translates of a two-branch hyperbola. Every horizontal slice is one interval of length $2\varepsilon$; every vertical slice consists of two intervals of total length exactly $2\varepsilon$, by Boole's lemma.}
		\label{fig:boole}
	\end{figure}

	The construction admits several extensions, all proved by the same two-step slice computation.

	\begin{proposition}\label{Pgeneral}
		Let $\psi,\phi$ belong to the singular Herglotz class with linear coefficients $\beta_\psi,\beta_\phi>0$, and let $E$ be measurable with $0<|E|<\infty$. Then
		$$S=\{(x,y):\ \psi(x)-\phi(y)\in E\}$$
		satisfies $\H\chi_S=\chi_S$, with horizontal and vertical cross-sections of constant measures $|E|/\beta_\psi$ and $|E|/\beta_\phi$. Diagonal sets are the case where $\psi$ and $\phi$ are affine. The same conclusion holds for suitable genuinely two-variable substitutions, for instance
		$$u(x,y)=x-y-\frac{c}{x-2y}\qquad(c>0),\qquad S=\{u\in E\}:$$
		for each fixed $y$ the map $x\mapsto u(x,y)$ and for each fixed $x$ the map $y\mapsto -u(x,y)$ belong to the singular Herglotz class, and the proof of Theorem \ref{Tbent} applies in each variable.
	\end{proposition}

	Two members of the family are shown in Figure \ref{fig:gallery}.

	\begin{figure}[htb]
		\centering
		\includegraphics[width=\textwidth]{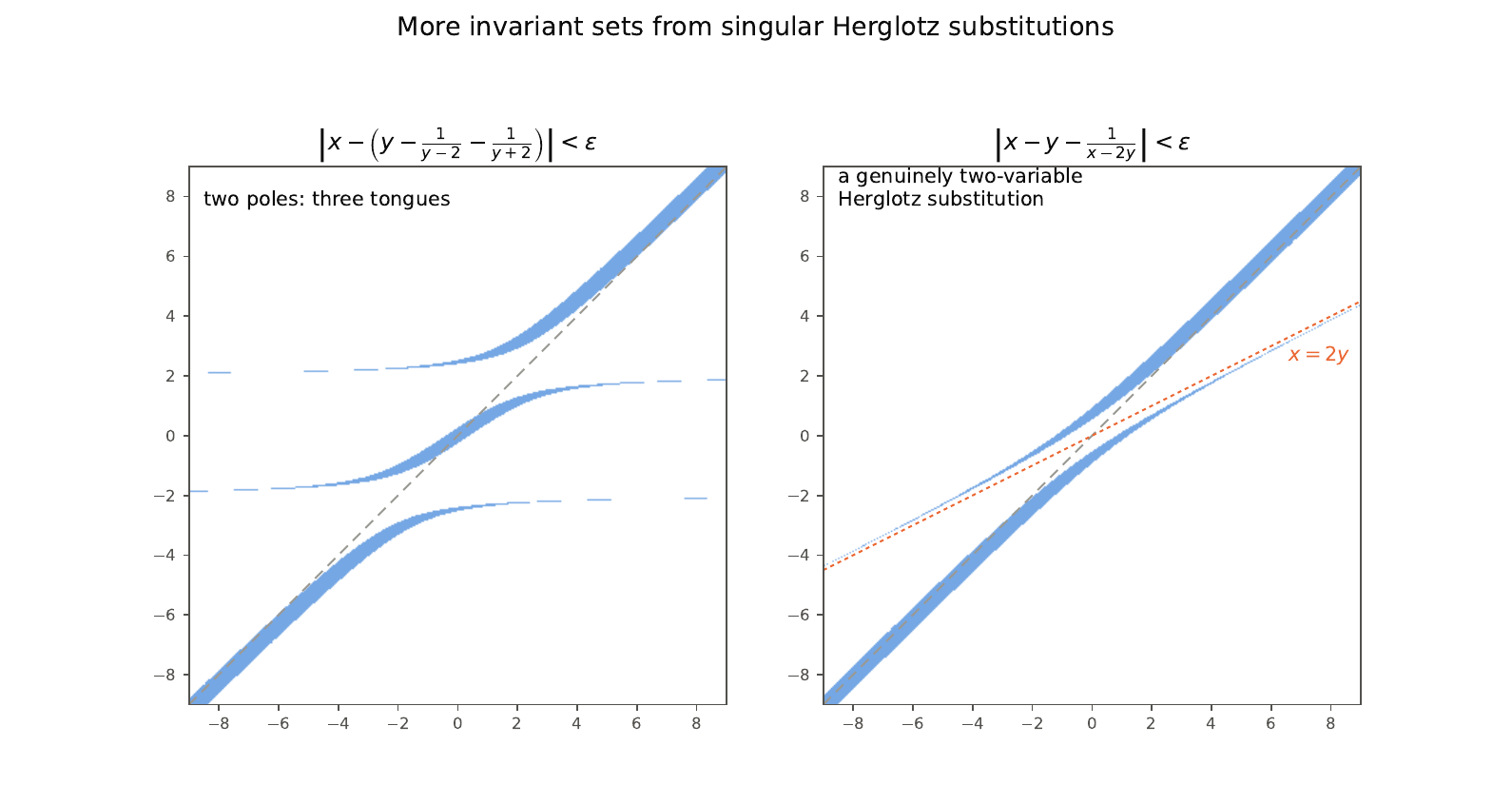}
		\caption{Two further invariant sets from Proposition \ref{Pgeneral}. Left: the two-pole substitution $\phi(y)=y-\frac{1}{y-2}-\frac{1}{y+2}$ produces three tongues. Right: the genuinely two-variable substitution $u=x-y-\frac{1}{x-2y}$; the thin leaf along the line $x=2y$ is part of the set, the two-variable analogue of the hyperbola's tongues.}
		\label{fig:gallery}
	\end{figure}

	\begin{example}[disjoint unions]\label{Egap}
		The class of invariant sets is closed under disjoint unions, and disjointness of two bent families is possible, though delicate. Let $0<\delta_0<\varepsilon'$, let $E_1$ be an interval of length $\delta_0$ and $E_2=E_1-\varepsilon'$, and let
		$$\phi_1(y)=y-\frac{\varepsilon'}{y+1},\qquad \phi_2(y)=y-\frac{\varepsilon'}{y-1}.$$
		Then $\delta:=\phi_1-\phi_2=2\varepsilon'/(y^2-1)$ omits the interval $(-2\varepsilon',0)$, while the set of shifts $c$ with $|(E_1+c)\cap E_2|>0$ is contained in $(-\varepsilon'-\delta_0,-\varepsilon'+\delta_0)\subset(-2\varepsilon',0)$. Hence the two bent strips $\{x-\phi_i(y)\in E_i\}$ are disjoint and their union is an invariant set whose horizontal slices consist of two intervals separated by a gap of variable length. In particular the union is not of the separated form of Proposition \ref{Pgeneral}. A computation with the ranges of $\delta$ shows that such disjointness requires $\beta_{\phi_1}=\beta_{\phi_2}$.
	\end{example}

	\section{Towards a classification of invariant sets}\label{sec:class}

	\subsection{The slice criterion and two necessary conditions}

	Throughout this section $S_y=\{x:(x,y)\in S\}$ and $S^x=\{y:(x,y)\in S\}$ denote the horizontal and vertical cross-sections, and
	$$F(t,y)=\widehat{\chi_{S_y}}(t)=\int_{S_y}e^{-ixt}\,dx .$$
	We interpret the invariance spectrally: $(\H-\mathcal I)\chi_S=0$ \emph{in the spectral sense} means that the tempered distribution $\widehat{\chi_S}$ is supported in $Q_{2,4}\cup\{0\}$, i.e.\ it vanishes on the open quadrants $Q_{1,3}$ (equivalently, $\mathcal P^1_+\mathcal P^2_+\chi_S=\mathcal P^1_-\mathcal P^2_-\chi_S=0$) and also on the punctured coordinate axes, where the symbol is discontinuous and where $\H$ in fact annihilates: functions of one variable alone are killed by $\H$, as are constants, which is why the origin is exempt. For sets with cross-sections of finite measure the slice-wise and the spectral interpretations agree; for the other examples of this note the relation between the two is discussed in Remark \ref{Rbmo}.

	\begin{remark}[what $\H$ means on $L^\infty$]\label{Rbmo}
		For sets with cross-sections of infinite measure -- the positive-density examples and the cones below -- the slice-wise definition breaks down at the first step: if $S_y$ contains a half-line, $\HH_1\chi_S(\cdot,y)$ is the Hilbert transform of an indicator of infinite measure, and its truncations diverge. The divergence, however, is a constant in $x$:
		$$\frac1\pi\int_a^R\frac{dt}{x-t}=\frac1\pi\log|x-a|-\frac1\pi\log(R-x)\qquad(a<x<R),$$
		so $\HH_1\chi_S(\cdot,y)$ exists modulo constants, i.e.\ as an element of $BMO(\R)$, and $\HH_1\chi_S$ exists modulo functions of $y$ alone. Applying $\HH_2$ produces a function modulo functions of $x$ alone, plus the transform of a function of $y$ alone. Altogether $\H$ maps $L^\infty(\R^2)$ into the product $BMO$ space of Chang and Fefferman \cite{CF}, whose elements are defined modulo $a(x)+b(y)$. On the Fourier side functions of one variable are exactly the distributions carried by the coordinate axes, and equality $\H f=f$ in product $BMO$ -- tested against Schwartz functions whose Fourier transforms vanish near both axes -- is exactly the vanishing of $\hat f$ on the open quadrants $Q_{1,3}$. The spectral interpretation is therefore not a convenient convention but the only meaning that $\H\chi_S=\chi_S$ can have for such sets. Concretely, for one wave with $m>0$,
		\begin{align*}
		\HH_1[\tmop{sign}(y-mx)]&=-\frac2\pi\log|y-mx|+b(y),\\
		\HH_2\big[\log|y-mx|\big]&=-\frac\pi2\,\tmop{sign}(y-mx)+a(x),
		\end{align*}
		the first identity being $\HH\,\tmop{sign}=\frac2\pi\log|\cdot|$ modulo constants, the second the statement that $\log|t|$ and $\pi\chi_{(-\infty,0)}(t)$ are the boundary values of $\log w$ on $\R$; hence $\H[\tmop{sign}(y-mx)]=\tmop{sign}(y-mx)$ modulo $a(x)+b(y)$, and likewise for finite sums of such waves. For the strips of Section \ref{sec:bent} the slices have finite measure and nothing of this is needed; for the periodic sets of Theorem \ref{Tcarry} symmetric truncations converge and reproduce the Fourier-series computation given there.

		When the set has a center of symmetry the ambiguity can be removed by a canonical normalization. Let $\psi$ be a smooth radial cutoff, $\psi=1$ near $0$, and $\psi_R=\psi(\cdot/R)$.
	\end{remark}

	\begin{lemma}[cutoff normalization]\label{Lcut}
		Let $G\in L^\infty(\R^2)$ be homogeneous of degree $0$ and odd about a point $p$, with $\hat G$ supported in $Q_{2,4}\cup\{0\}$. Then $\H[G\,\psi_R(\cdot-p)]\to G$ as $R\to\infty$, uniformly on compact sets and in $\mathcal S'(\R^2)$.
	\end{lemma}

	\begin{proof}
		Take $p=0$ and let $\phi$ be the symbol of $\H$. Since $G\psi_R\in L^2$, the function $\H(G\psi_R)-G\psi_R$ is the inverse Fourier transform of $e_R=(\phi-1)(\hat G*\hat\psi_R)$. Homogeneity of $\hat G$ (of degree $-2$) gives $\hat G*\hat\psi_R=R^2(\hat G*\hat\psi)(R\,\cdot)$, and $\phi$ is dilation invariant, so $e_R=R^2e(R\,\cdot)$ with $e=(\phi-1)(\hat G*\hat\psi)$. Now $e$ is bounded and supported in $\overline{Q_{1,3}}$, where it decays rapidly, because $\hat\psi$ is a Schwartz function and $\hat G$ is carried by a closed cone in $Q_{2,4}\cup\{0\}$, at positive angular distance from $Q_{1,3}$; hence $e\in L^1$. Since $\hat G$ is odd while $\hat\psi$ and $\phi$ are even, $e$ is odd and $\int e=0$. Therefore $\H(G\psi_R)-G\psi_R=\check e(\cdot/R)$ with $\check e$ continuous and $\check e(0)=0$, which tends to $0$ uniformly on compact sets and boundedly; and $G\psi_R=G$ on $\{|z|<cR\}$.
	\end{proof}

	Without oddness the same argument gives $\H(G\psi_R)\to G+c$ with a constant $c=\int e$; in the same spirit, an atom of $\hat f$ at the origin -- a nonzero mean -- is annihilated by any such normalization, which is the constant $\frac12$ in Theorem \ref{Tcarry}(iii). Numerically, for the fans of Proposition \ref{Pfan} the error $\H(G\psi_R)-G$ on a fixed window decreases like $1/R$, while for functions whose spectrum meets $Q_{1,3}$ it stays of order one.

	\begin{lemma}[slicing]\label{Lslice}
		Let $S$ be measurable with $M:=\sup_y|S_y|<\infty$ and let $\widehat{\chi_S}$ vanish on $Q_{1,3}$. Then for every $t>0$ the function $F(t,\cdot)$ coincides a.e.\ with the boundary values of a function analytic in $\C_-$ and bounded by $M$; for $t<0$, the same holds in $\C_+$.
	\end{lemma}

	\begin{proof}
		We have $|F(t,y)|\leq|S_y|\leq M$, and $t\mapsto F(t,y)$ is continuous by dominated convergence on the finite-measure slice. Let $\alpha,\beta\in C^\infty_c((0,\infty))$. Then, by Fubini (all integrals converge absolutely; we omit an irrelevant positive constant),
		$$0=\langle\widehat{\chi_S},\alpha\otimes\beta\rangle=\iint\chi_S(x,y)\hat\alpha(x)\hat\beta(y)\,dxdy=\int_0^\infty\alpha(t)\,G_\beta(t)\,dt,$$
		where $G_\beta(t)=\int_\R F(t,y)\hat\beta(y)\,dy$.
		Since $G_\beta$ is continuous, $G_\beta\equiv 0$ on $(0,\infty)$; that is, for each fixed $t>0$ the distributional Fourier transform of $F(t,\cdot)$ annihilates every test function supported in $(0,\infty)$, so its support lies in $(-\infty,0]$. Now let $u$ be the Poisson extension of $F(t,\cdot)$ to $\C_-$; it is bounded by $M$. On the Fourier side the Poisson multiplier $e^{-|v||\xi|}$ coincides, on the support $\{\xi\leq 0\}$, with the multiplier $e^{-v\xi}$ of the analytic extension of $e^{iy\xi}$ to $w=y+iv$, $v<0$; a direct computation then gives $\partial_{\bar w}u=0$. Hence $u\in H^\infty(\C_-)$ with boundary values $F(t,\cdot)$. The case $t<0$ follows since $F(-t,y)=\overline{F(t,y)}$.
	\end{proof}

	\begin{proposition}[constancy of cross-sections]\label{Pconst}
		Under the assumptions of Lemma \ref{Lslice}, $|S_y|$ is a.e.\ constant. Symmetrically, if the vertical cross-sections are uniformly bounded, $|S^x|$ is a.e.\ constant.
	\end{proposition}

	\begin{proof}
		As $t\to 0^+$, $F(t,y)\to|S_y|$ pointwise, boundedly. The space of boundary values of $H^\infty(\C_-)$ is weak-star closed in $L^\infty(\R)$, so $|S_\cdot|$ is such a boundary value; being real a.e., the imaginary part of its analytic extension is a bounded harmonic function with vanishing boundary values, hence zero, and the extension is a real constant.
	\end{proof}

	\begin{proposition}[barycenters are Herglotz]\label{Pbary}
		Assume in addition that $y\mapsto\int_{S_y}|x|\,dx$ is locally bounded. Then the first moment $m_1(y)=\int_{S_y}x\,dx$ coincides a.e.\ with the boundary values of an analytic map $G:\C_-\to\{\Im w\leq 0\}$. Being real a.e., $m_1$ is (the boundary value of) a real Nevanlinna function with purely singular representing measure. Symmetrically for vertical slices.
	\end{proposition}

	\begin{proof}[Proof (sketch)]
		Let $a=|S_y|$ (constant). The analytic extension of $F(t,\cdot)$ satisfies $|F(t,w)|\leq a$ in $\C_-$, so $a-F(t,w)$ has nonnegative real part and $G_t:=(a-F(t,\cdot))/(it)$ maps $\C_-$ into $\{\Im\leq 0\}$. On the boundary,
		$$G_t(y)=\int_{S_y}\frac{1-e^{-ixt}}{it}\,dx\longrightarrow m_1(y)\qquad(t\to 0^+),$$
		dominated by $\int_{S_y}|x|dx$; the same bound dominates the Poisson extensions, so $(G_t)$ is a normal family on $\C_-$ and any locally uniform limit $G$ is analytic with $\Im G\leq 0$ and boundary values $m_1$ a.e.
	\end{proof}

	Note that for the bent strip of Theorem \ref{Tbent} one has $m_1(y)=|E|\,\phi(y)+c$: the proposition recovers exactly the Herglotz datum of the example, and shows that the singular Herglotz class is forced on every invariant set, not chosen by the construction.

	\subsection{Sets with connected cross-sections}

	\begin{theorem}[classification: interval slices]\label{Tinterval}
		Let $S$ have all cross-sections of finite measure, horizontal cross-sections uniformly bounded, and suppose that a.e.\ horizontal slice is (up to a null set) an interval. Then $\H\chi_S=\chi_S$ if and only if
		$$S=\{(x,y):\ 0<x-\phi(y)<a\}\quad\text{up to a null set},$$
		for some $a>0$ and some $\phi$ of the singular Herglotz class \eqref{herglotz} with $\beta>0$ and purely singular $\mu$. Diagonal sets correspond exactly to $\mu=0$.
	\end{theorem}

	\begin{proof}[Proof (sketch)]
		Sufficiency is Theorem \ref{Tbent}. For necessity: by Proposition \ref{Pconst} the slices are intervals $(\alpha(y),\alpha(y)+a)$ of constant length, so
		$$F(t,y)=e^{-it\alpha(y)}\,\frac{1-e^{-ita}}{it}.$$
		For every $t>0$ outside the discrete set $\frac{2\pi}a\Z$, Lemma \ref{Lslice} shows that the unimodular function $e^{-it\alpha(\cdot)}$ is the boundary value of a function $\Theta_t\in H^\infty(\C_-)$, i.e.\ an inner function of $\C_-$. Boundary values determine inner functions, so $\Theta_t\Theta_{t'}=\Theta_{t+t'}$ whenever defined, and $t\mapsto\Theta_t$ is weak-star continuous; hence $(\Theta_t)_{t>0}$ extends to a continuous semigroup of inner functions. A semigroup element has no zeros (any zero of $\Theta_t=\Theta_{t/n}^n$ would have order divisible by every $n$), so, $\C_-$ being simply connected, $\Theta_t=e^{-it\Phi}$ for a single analytic $\Phi:\C_-\to\overline{\C_-}$. Unimodularity of the boundary values forces $\Im\Phi=0$ a.e.\ on $\R$, i.e.\ $\Phi$ is a real Nevanlinna function with purely singular measure, and its boundary values equal $\alpha$ a.e. Finally, $\beta>0$: the vertical cross-sections are the sets $\Phi^{-1}(x-(0,a))$, and for $\beta=0$ the push-forward of Lebesgue measure under a real singular Nevanlinna function is never locally finite (this follows from the disintegration of Lebesgue measure into the Clark measures of $\Phi$ together with the asymptotics of $\Phi(iv)$ as $v\to\infty$; we omit the details), so that some vertical cross-section would have infinite measure.
	\end{proof}

	\begin{remark}
		The same conclusion holds if the slices are only assumed to be translates $S_y=\alpha(y)+E$ of one fixed set $E$ of finite measure. Indeed, then $F(t,y)=e^{-it\alpha(y)}\widehat{\chi_E}(t)$, and since $\widehat{\chi_E}$ is continuous with $\widehat{\chi_E}(0)>0$, the function $e^{-it\alpha(\cdot)}$ is an inner boundary value for all small $t>0$, hence, by multiplicativity, for all $t>0$, and the semigroup argument applies verbatim.
	\end{remark}

	\subsection{Positive density: the three-family example, fans and cones}

	The examples above have cross-sections of finite measure. At positive density the mean becomes an obstruction: the symmetric principal value of $\int c\,(x-t)^{-1}dt$ vanishes for every constant $c$, so $\H$ annihilates means slice-wise, and no set of positive density satisfies $\H\chi_S=\chi_S$ exactly. The correct normalization is
	$$\H(2\chi_S-1)=2\chi_S-1 .$$

	The basic computation is the following: for $m>0$ and a $1$-periodic $f\in L^2_{loc}(\R)$ with $\int_0^1f=0$, expanding $f(y-mx)=\sum_{p\neq 0}c_pe^{2\pi ip(y-mx)}$, the wave $e^{2\pi ip(y-mx)}$ has frequency $(-2\pi pm,\,2\pi p)$, which lies in $Q_{2,4}$, where the symbol of $\H$ equals $(i\,\tmop{sign}\, p)(-i\,\tmop{sign}\,p)=1$; hence
	\begin{equation}\label{waves}
		\H\left[f(y-mx)\right]=f(y-mx)\qquad(m>0,\ \textstyle\int_0^1 f=0).
	\end{equation}

	\begin{theorem}[the three-family set]\label{Tcarry}
		Let $u=y-x$, $v=y-\frac x2$ and
		$$S=\left\{(x,y):\ \{u\}+\{v\}>1\right\},$$
		where $\{\cdot\}$ denotes the fractional part. Then:

		(i) $\chi_S=\{u\}+\{v\}-\{u+v\}$ (the carry identity);

		(ii) every horizontal and every vertical line meets $S$ in a set of density exactly $1/2$;

		(iii) $\H\chi_S=\chi_S-\frac 12$, equivalently $\H(2\chi_S-1)=2\chi_S-1$.
	\end{theorem}

	\begin{proof}
		(i) is the statement that $\{u\}+\{v\}-\{u+v\}\in\{0,1\}$ equals $1$ exactly when $\{u\}+\{v\}>1$, i.e.\ when a carry occurs in the addition $u+v$. For (iii), each of the three sawtooths $\{u\}-\frac12$, $\{v\}-\frac 12$, $\{u+v\}-\frac12$ is a mean-zero periodic wave constant along lines of slope $1$, $\frac 12$ and $\frac 34$ respectively (note $u+v=2y-\frac 32 x$), so \eqref{waves} applies to each; the constant $\frac 12$ is annihilated. (ii) follows from (i): $\chi_S-\frac12=(\{u\}-\frac12)+(\{v\}-\frac12)-(\{u+v\}-\frac12)$, and along any horizontal or vertical line each of the three terms is a mean-zero periodic function of the parameter (with periods $1$, $2$, $\frac23$ along horizontal lines and $1$, $1$, $\frac12$ along vertical ones), so the mean of $\chi_S$ along the line is $\frac12$.
	\end{proof}

	\begin{remark}[geometry]\label{Rgeom}
		$S$ is bounded by the three families of parallel lines
		$$u\in\Z,\qquad v\in\Z,\qquad u+v\in\Z,$$
		that is, $y=x+n$, $y=\frac x2+n$ and $y=\frac 34x+\frac n2$, $n\in\Z$. The cells of the first two families are parallelograms whose four vertices lie at four consecutive integer heights; the third family passes through each cell as its long diagonal and divides it into two halves of area $1$. The set $S$ is the union of the upper halves. The complement of $S$ is a set of the same type, by $\{-t\}=1-\{t\}$. On the Fourier side, all nonzero frequencies of $2\chi_S-1$ lie on the three lines $\xi_2/\xi_1\in\{-1,-2,-\frac 43\}$ in $Q_{2,4}$: the example is a sum of three ``diagonal'' generalized eigenvectors which happens to be two-valued. See Figure \ref{fig:carry}.
	\end{remark}

	The three-family set has a degenerate relative in which the three families collapse to lines through one point. Write $H=\chi_{[0,\infty)}$ for the Heaviside function. We say that affine functions $a_dy-b_dx+c_d$ are \emph{oriented alike} if the coefficients $a_d$ all have the same sign.

	\begin{proposition}[fans]\label{Pfan}
		Let $k$ be odd and let $\ell_1,\dots,\ell_k$ be affine functions on $\R^2$, oriented alike, whose zero lines pass through a common point $p$ and have distinct positive slopes, labeled in increasing order of slope. Put
		$$G=\sum_{d=1}^k(-1)^{d-1}\tmop{sign}\,\ell_d,\qquad F=\{G=1\}.$$
		Then:

		(i) $G$ is $\{\pm1\}$-valued a.e., so $G=2\chi_F-1$, and $F$ is the union of $k$ of the $2k$ sectors into which the lines cut the plane, taken alternately; replacing every $\ell_d$ by $-\ell_d$ replaces $F$ by its complement. Conversely, every union of alternate sectors of $k$ concurrent lines of positive slopes is of this form, and $k$ is necessarily odd for such a union to be invariant.

		(ii) $2\chi_F-1$ is odd with respect to $p$ and homogeneous of degree $0$ about $p$, and $\H(2\chi_F-1)=2\chi_F-1$ in the spectral sense; the spectrum consists of the $k$ lines through the origin normal to the boundary lines.

		(iii) For $k=3$, and $\ell_3=\lambda_1\ell_1+\lambda_2\ell_2$ with $\lambda_1,\lambda_2>0$ (which is automatic after relabelling, the middle slope being a weighted mean of the other two), $F=\{H(\ell_1)+H(\ell_2)-H(\ell_3)=1\}$. At every vertex $p$ of its arrangement the carry set $S$ of Theorem \ref{Tcarry} blows up to such a fan: as $\varepsilon\to0^+$, $\chi_S(p+\varepsilon(z-p))\to\chi_{F}(z)$ a.e., where $F$ is built on $\ell_1=u(p)-u$, $\ell_2=v(p)-v$, $\ell_3=\ell_1+\ell_2$.
	\end{proposition}

	\begin{proof}
		(i) Since the functionals are oriented alike and the slopes are positive, the gradients $\nabla\ell_d=(-b_d,a_d)$ lie in one open quadrant, in the angular order of the slopes. Walking around $p$ one therefore crosses the $k$ lines first in the increasing direction of their functionals, in the order $d=1,\dots,k$, and then in the decreasing direction, in the same order. The jumps of $G$ along this walk are $2(-1)^{d-1}$ at the increasing crossings and $-2(-1)^{d-1}$ at the decreasing ones: the cyclic sequence of jumps is $+2,-2,\dots,+2,-2,+2,\dots,-2$, alternating, precisely because $k$ is odd. In the sector where all $\ell_d<0$ one has $G=-\sum(-1)^{d-1}=-1$, so $G$ alternates between $-1$ and $+1$. For even $k$ the two junctions of the walk carry consecutive jumps of the same sign, and $G$ takes three values. Given an alternate union $F$ of sectors of $k$ concurrent positive-slope lines, define $G$ as above with the labelling by slopes; $G$ and $2\chi_F-1$ are both odd about $p$ (an alternate coloring of $2k$ sectors is antipodally antisymmetric) and have the same jumps, hence coincide. Invariance of such an $F$ forces $k$ odd by Lemma \ref{Lparity} below, or by Theorem \ref{Tcone}. The statement about $-\ell_d$ is clear.

		(ii) Oddness and homogeneity follow from $\ell_d(p+\lambda(z-p))=\lambda\ell_d(z)$. In the coordinate $s=\ell_d(z)$ the function $\tmop{sign}\,\ell_d$ is $\tmop{sign}(s)\otimes 1$, so its Fourier transform is a phase factor (accounting for $p$) times $\pv\,1/\sigma$ carried by the line $\R\nabla\ell_d$ through the origin, which lies in $Q_{2,4}\cup\{0\}$ because $a_db_d>0$. Hence the Fourier transform of $2\chi_F-1$ is supported in $Q_{2,4}\cup\{0\}$, which is the spectral invariance. (By Remark \ref{Rbmo} this is the only possible meaning of the identity slice-wise, up to functions of one variable, and by Lemma \ref{Lcut} it holds exactly after the cutoff normalization: $\H[(2\chi_F-1)\psi_R(\cdot-p)]\to2\chi_F-1$ uniformly on compact sets.)

		(iii) With $\tmop{sign}\,\ell=2H(\ell)-1$ a.e., $G=1$ reads $H(\ell_1)-H(\ell_2)+H(\ell_3)=1$ when the middle slope is $\ell_2$; relabeling the middle functional as $\ell_3$ gives the stated form, which is also verified directly: if $\ell_1,\ell_2\ge0$ then $\ell_3\ge0$ and the sum is $1$; if both are negative it is $0$; if they have opposite signs it is $1-H(\ell_3)$. For the blow-up, near a vertex $p$ the carry set is bounded by the three lines through $p$ only. Put $m_1=u-u(p)$, $m_2=v-v(p)$; since $u(p),v(p)\in\Z$, for $m_1,m_2$ small one has $\{u\}=m_1+1-H(m_1)$, $\{v\}=m_2+1-H(m_2)$ and $\{u+v\}=m_1+m_2+1-H(m_1+m_2)$, so that
		\begin{align*}
		\{u\}+\{v\}-\{u+v\}&=1-H(m_1)-H(m_2)+H(m_1+m_2)\\
		&=H(-m_1)+H(-m_2)-H(-m_1-m_2)
		\end{align*}
		a.e., which is the fan built on $\ell_1=-m_1$, $\ell_2=-m_2$ (oriented alike, with coefficient $-1$ of $y$).
	\end{proof}

	\begin{figure}[htb]
		\centering
		\includegraphics[width=.82\textwidth]{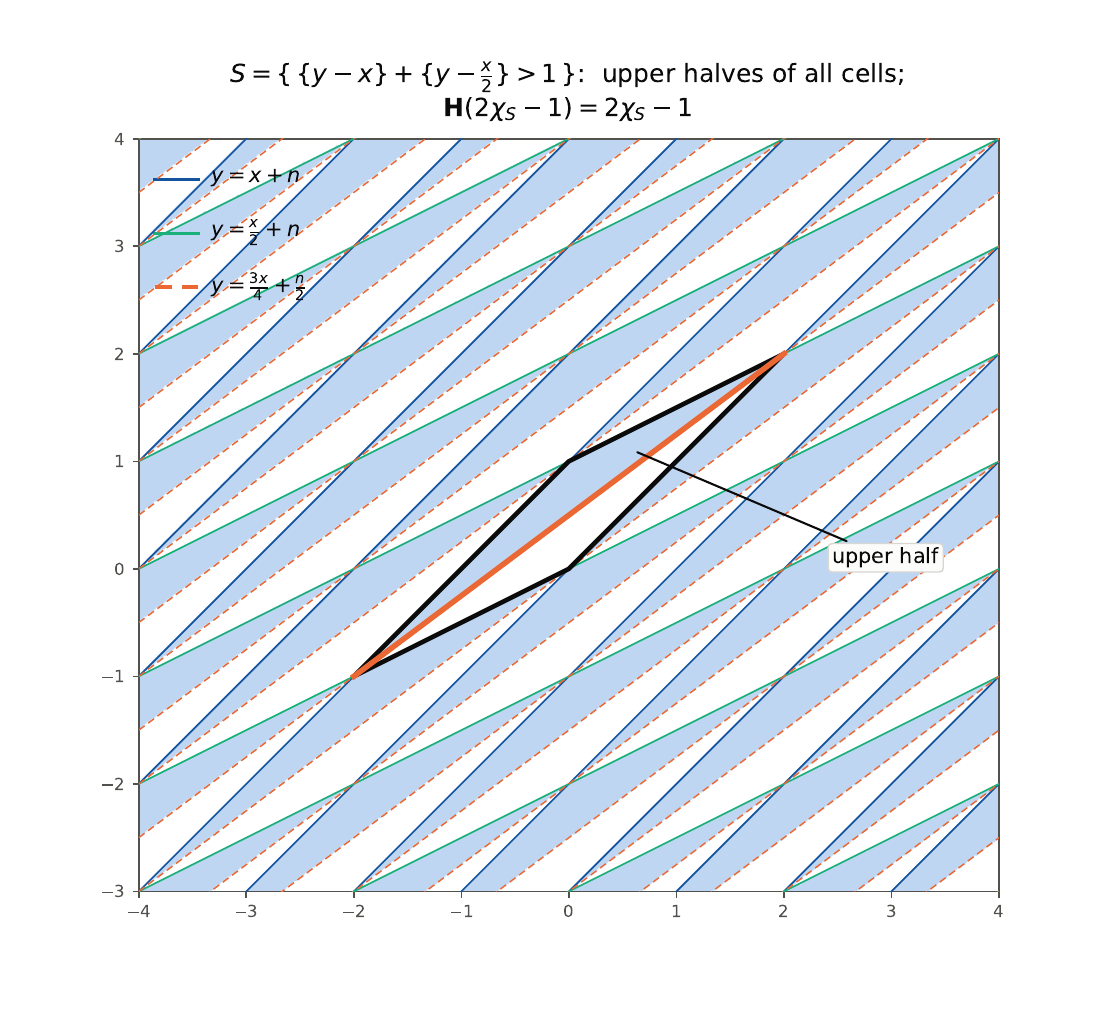}
		\caption{The three-family set of Theorem \ref{Tcarry}. The cells of the families $y=x+n$ and $y=\frac x2+n$ are parallelograms; the family $y=\frac34x+\frac n2$ passes through each cell as its long diagonal, and $S$ (shaded) is the union of the upper halves. One cell and its diagonal are highlighted.}
		\label{fig:carry}
	\end{figure}

	The fans are the simplest invariant sets with a nontrivial arrangement: $k$ concurrent lines, $2k$ alternating sectors, a single point of odd multiplicity (Figure \ref{fig:fans}). Their profiles are Heaviside steps rather than sawtooths; we return to this distinction in Proposition \ref{Pchar}. Note that for $k=3$ the slope of the sum family is intermediate only when $\ell_1,\ell_2$ are oriented alike: a fan built on $y-x$ and $4x-2y$, with sum $3x-y$, has slopes $1,2,3$, but it coincides with the fan built on the like-oriented pair $y-x$, $y-3x$, whose sum $2y-4x$ has the middle slope. The same remark applies to the three-family set: its families can be relabeled so that the resonant one is the sum of two like-oriented functionals, and then its slope is the intermediate one.

	\begin{figure}[htb]
		\centering
		\includegraphics[width=\textwidth]{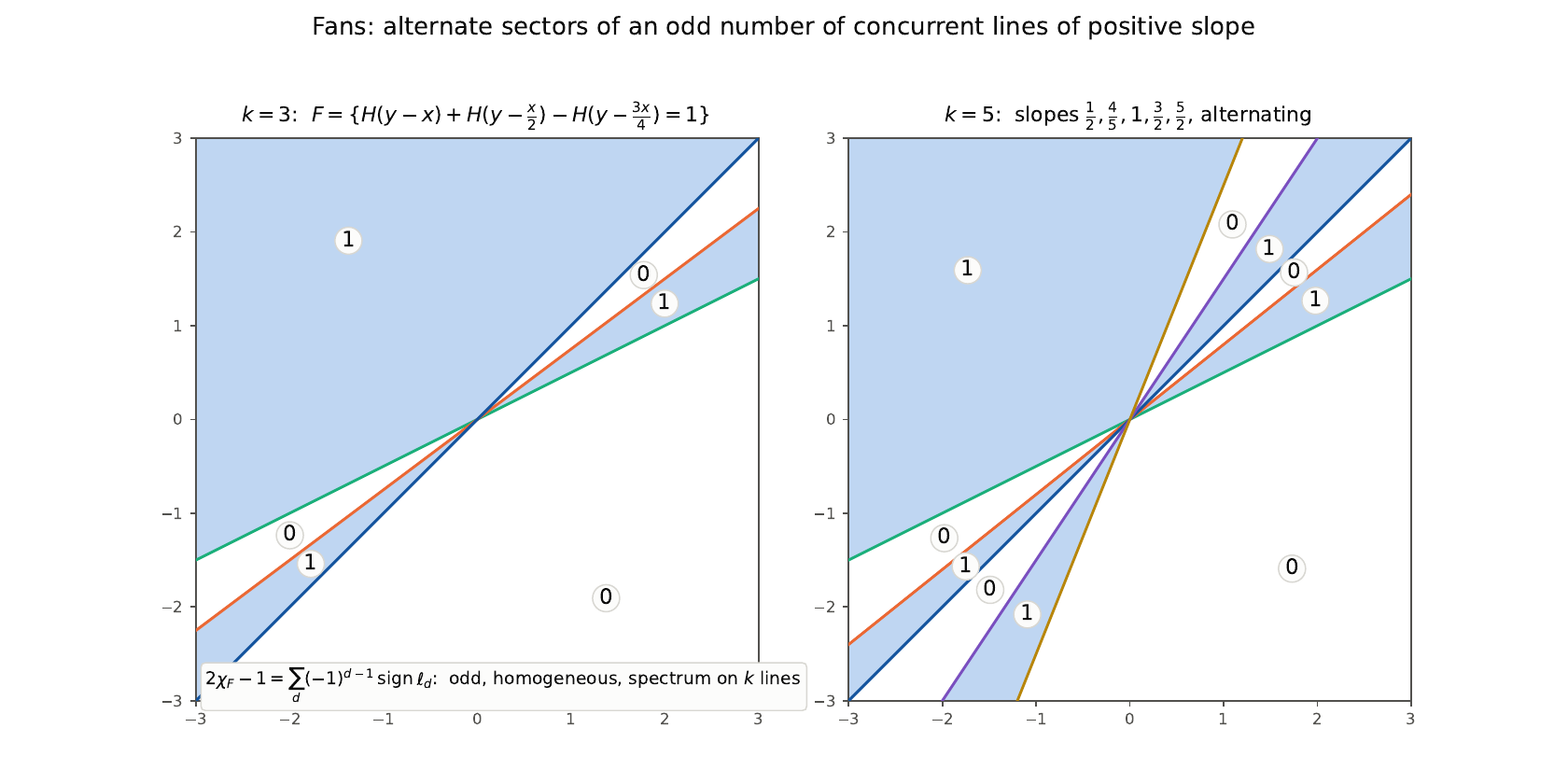}
		\caption{Two fans of Proposition \ref{Pfan}: three and five concurrent lines of positive slope, with the alternate sectors forming $F$ (shaded). The three-line fan is the blow-up of the three-family set at any vertex of its arrangement; the five-line fan shows that, with step profiles, five directions can occur.}
		\label{fig:fans}
	\end{figure}

	Fans are the piecewise constant members of a much larger family, which can be described completely. Identify directions with angles $\theta\in\T=\R/2\pi\Z$ and call $\theta$ \emph{of positive slope} if $\cos\theta\sin\theta>0$.

	\begin{theorem}[conical invariant sets]\label{Tcone}
		Let $F=\{p+re^{i\theta}:\ r>0,\ \theta\in A\}$ be a cone with vertex $p$ over a measurable set $A\subset\T$ which is neither null nor conull, and let $\Omega=2\chi_A-1$. Then $\H(2\chi_F-1)=2\chi_F-1$ in the spectral sense if and only if

		(a) $\Omega$ is odd: $\Omega(\theta+\pi)=-\Omega(\theta)$ a.e., and

		(b) $\Omega$ is locally constant on a neighborhood of the two closed arcs of directions of non-positive slope.

		\noindent In that case the Fourier transform of $2\chi_F-1$ is, up to a phase factor and a constant, the homogeneous distribution $|\xi|^{-2}\,\Omega'(\varphi-\pi/2)$, $\xi=|\xi|e^{i\varphi}$, carried by the lines through the origin normal to the rays on which $\Omega$ is not locally constant, and $\H[(2\chi_F-1)\psi_R(\cdot-p)]\to2\chi_F-1$ uniformly on compact sets by Lemma \ref{Lcut}. Thus $A$ is determined by its trace on one open arc of positive slope, where it is an arbitrary measurable set (subject to (b) near the endpoints), and $F$ is a fan exactly when this trace is a finite union of intervals.
	\end{theorem}

	\begin{proof}
		Throughout, $z=re^{i\theta}$ and $\xi=\rho e^{i\varphi}$ are polar coordinates in the physical and in the frequency plane, $\T=\R/2\pi\Z$,
		$$P=\{\theta:\ \cos\theta\sin\theta>0\},\qquad N=\{\varphi:\ \cos\varphi\sin\varphi<0\}$$
		are the open sets of directions of positive, respectively negative, slope, and $\tau_aT=T(\cdot-a)$ denotes the rotation of a distribution $T$ on $\T$ by the angle $a$, so that $\tmop{supp}(\tau_aT)=\tmop{supp}T+a$; note that $\tau_{\pi/2}$ maps $P$ onto $N$. A translation of $F$ multiplies the Fourier transform by a unimodular character and does not affect its support, so we may take $p=0$; then $G:=2\chi_F-1=\Omega(\theta)$ is homogeneous of degree $0$.

		\emph{Step 1: angular harmonics.} For $k\in\Z\setminus\{0\}$ put $E_k(z)=e^{ik\theta}$ and define a tempered distribution $D_k$ by
		$$\langle D_k,\psi\rangle=\int_0^\infty\psi_k(\rho)\,\frac{d\rho}{\rho},\qquad \psi_k(\rho)=\int_0^{2\pi}e^{ik\varphi}\psi(\rho e^{i\varphi})\,d\varphi\qquad(\psi\in\mathcal S(\R^2)).$$
		The integral converges. Indeed, every $\varphi$-derivative of $\psi(\rho e^{i\varphi})$ is a sum of terms $(\partial^\beta\psi)(\rho e^{i\varphi})$ times a homogeneous polynomial of degree $|\beta|\geq1$ in $\rho\cos\varphi,\rho\sin\varphi$, so integrating by parts $M\geq1$ times in $\varphi$ gives
		\begin{equation}\label{psik}
			|\psi_k(\rho)|\leq C_M(\psi)\,|k|^{-M}\min\{\rho,(1+\rho)^{-2}\},
		\end{equation}
		with $C_M(\psi)$ a Schwartz seminorm of $\psi$. We claim that
		\begin{equation}\label{BH}
			\widehat{E_k}=c\,|k|\,i^{-|k|}\,D_k ,
		\end{equation}
		with a constant $c\neq0$ depending only on the normalization of the Fourier transform. To see this, write $E_k=P_k(z)|z|^{-|k|}$, where $P_k(z)=z^k$ for $k>0$ and $P_k(z)=\bar z^{|k|}$ for $k<0$ is a harmonic homogeneous polynomial of degree $|k|$ with $P_k(\rho e^{i\varphi})=\rho^{|k|}e^{ik\varphi}$. The Bochner--Hecke formula (\cite{SW2}, Chapter IV, Theorem 4.1) states that for $0<\alpha<2$
		$$\mathcal F\big[P_k(z)|z|^{-|k|-2+\alpha}\big]=\gamma_{k,\alpha}\,P_k(\xi)|\xi|^{-|k|-\alpha},\qquad\gamma_{k,\alpha}=c_\alpha\, i^{-|k|}\,\frac{\Gamma\big(\frac{|k|+\alpha}2\big)}{\Gamma\big(\frac{|k|+2-\alpha}2\big)},$$
		where $c_\alpha>0$ depends only on $\alpha$ and on the normalization and is continuous in $\alpha\in(0,2]$. Let $\alpha\uparrow2$. On the left, $r^{\alpha-2}e^{ik\theta}\to e^{ik\theta}$ in $\mathcal S'$ by dominated convergence, since $r^{\alpha-2}\leq r^{-1}+1$ is locally integrable in the plane. On the right, for $\psi\in\mathcal S$,
		$$\big\langle P_k(\xi)|\xi|^{-|k|-\alpha},\psi\big\rangle=\int_0^\infty\rho^{1-\alpha}\psi_k(\rho)\,d\rho\longrightarrow\int_0^\infty\psi_k(\rho)\,\frac{d\rho}\rho=\langle D_k,\psi\rangle,$$
		again by dominated convergence, \eqref{psik} providing the majorant $\rho^{1-\alpha}|\psi_k(\rho)|\leq C\min\{1,(1+\rho)^{-2}\}$ for $1\leq\alpha<2$. Since $\gamma_{k,\alpha}\to c_2\,i^{-|k|}\,\Gamma(\frac{|k|}2+1)/\Gamma(\frac{|k|}2)=c_2\,i^{-|k|}\,\frac{|k|}2$, this proves \eqref{BH} with $c=c_2/2$. (The distribution $D_k$ is what is meant by the principal value $\pv[e^{ik\varphi}|\xi|^{-2}]$: the integral is taken in the angular variable first, and it converges because the angular mean of $e^{ik\varphi}$ vanishes.)

		\emph{Step 2: the angular distribution.} For a distribution $T$ on $\T$ of finite order with $\langle T,1\rangle=0$ define $T\otimes\rho^{-2}\in\mathcal S'(\R^2)$ by
		$$\big\langle T\otimes\rho^{-2},\psi\big\rangle=\int_0^\infty\big\langle T,\psi(\rho e^{i\,\cdot})\big\rangle\,\frac{d\rho}\rho .$$
		The integrand is $O(\rho)$ as $\rho\to0$, because $\psi(\rho e^{i\varphi})=\psi(0)+\rho R(\rho,\varphi)$ with $R$ smooth and $T$ annihilates constants, and it is rapidly decreasing as $\rho\to\infty$, because $\|\psi(\rho e^{i\,\cdot})\|_{C^m(\T)}=O(\rho^{m-M})$ for every $M$; so the integral converges and depends continuously on $\psi$. For $T=e^{ik\varphi}$ one recovers $D_k$, and if $T=\sum_{k\neq0}t_ke^{ik\varphi}$ with $|t_k|\leq C|k|^s$, then by \eqref{psik} the series $\sum_kt_k\psi_k(\rho)$ converges absolutely, with a majorant integrable against $d\rho/\rho$, so that
		\begin{equation}\label{series}
			\big\langle T\otimes\rho^{-2},\psi\big\rangle=\sum_{k\neq0}t_k\langle D_k,\psi\rangle .
		\end{equation}
		Now expand $\Omega=\sum_k\Omega_ke^{ik\theta}$ in $L^2(\T)$. The partial sums $S_K\Omega(\theta)$ converge to $G$ in $\mathcal S'(\R^2)$, since $|\langle G-S_K\Omega(\theta),\psi\rangle|\leq\|\Omega-S_K\Omega\|_{L^2(\T)}\int_0^\infty\|\psi(\rho e^{i\,\cdot})\|_{L^2(\T)}\,\rho\,d\rho$. Taking Fourier transforms, \eqref{BH} and \eqref{series} (with $t_k=|k|i^{-|k|}\Omega_k$, $s=1$) give
		\begin{equation}\label{Ghat}
			\hat G=a\,\Omega_0\,\delta_0+c\,\big(T_\Omega\otimes\rho^{-2}\big),\qquad T_\Omega=\sum_{k\neq0}|k|\,i^{-|k|}\,\Omega_k\,e^{ik\varphi},
		\end{equation}
		where $a\,\delta_0$ is the transform of the constant $1$. To identify $T_\Omega$, split $\Omega=\Omega_o+\Omega_e$ into the sums over odd and over even $k$, so that $\Omega_o(\theta)=\frac12\big(\Omega(\theta)-\Omega(\theta+\pi)\big)$ and $\Omega_e(\theta)=\frac12\big(\Omega(\theta)+\Omega(\theta+\pi)\big)$. Since $i^{-|k|}=e^{-i\pi|k|/2}$, one checks that $|k|\,i^{-|k|}=k\,e^{-i\pi k/2}$ when $k$ is odd and $|k|\,i^{-|k|}=|k|\,e^{-i\pi k/2}$ when $k$ is even. Hence
		\begin{equation}\label{TOmega}
			\begin{aligned}
			T_\Omega(\varphi)&=\sum_{k\ \mathrm{odd}}k\,\Omega_k\,e^{ik(\varphi-\pi/2)}+\sum_{k\ \mathrm{even},\,k\neq0}|k|\,\Omega_k\,e^{ik(\varphi-\pi/2)}\\
			&=-i\,\tau_{\pi/2}\big(\Omega_o'\big)+\tau_{\pi/2}\big(|D|\Omega_e\big),
			\end{aligned}
		\end{equation}
		where $|D|$ is the Fourier multiplier $|k|$ on $\T$, i.e.\ $|D|=\mathcal H\partial_\varphi$ with $\mathcal H$ the conjugate function operator (multiplier $-i\,\tmop{sign}k$). Both $\Omega_o'$ and $|D|\Omega_e$ are distributions of finite order with zero mean, so \eqref{Ghat} applies to $\Omega_o$ and to $\Omega_e$ separately. The operator $\partial_\varphi$ is local, whereas $|D|$ is not; this is why odd and even $\Omega$ behave so differently.

		\emph{Step 3: supports.} For $T$ as in Step 2,
		\begin{equation}\label{supp}
			\tmop{supp}\big(T\otimes\rho^{-2}\big)\setminus\{0\}=\{\rho e^{i\varphi}:\ \rho>0,\ \varphi\in\tmop{supp}T\}.
		\end{equation}
		If $\psi\in C_c^\infty(\R^2)$ is supported away from the origin and from the right-hand side, then for every $\rho$ the function $\psi(\rho e^{i\,\cdot})$ is supported in $\T\setminus\tmop{supp}T$, so $\langle T,\psi(\rho e^{i\,\cdot})\rangle=0$; this proves the inclusion $\subseteq$. Conversely, let $\varphi_0\in\tmop{supp}T$, $\rho_0>0$, and let $W$ be a neighborhood of $\rho_0e^{i\varphi_0}$. Choose $a_0\in C_c^\infty((0,\infty))$ with $\int a_0\,d\rho/\rho\neq0$ and $b\in C^\infty(\T)$ with $\langle T,b\rangle\neq0$, supported so close to $\rho_0$ and $\varphi_0$ that $\psi(\rho e^{i\varphi})=a_0(\rho)b(\varphi)$ is supported in $W$; then $\langle T\otimes\rho^{-2},\psi\rangle=\langle T,b\rangle\int a_0\,d\rho/\rho\neq0$. Since the cone over $N$ is $Q_{2,4}$, \eqref{supp} yields the criterion
		\begin{equation}\label{crit}
			\tmop{supp}\big(T\otimes\rho^{-2}\big)\subset Q_{2,4}\cup\{0\}\iff\tmop{supp}T\subset N\iff\tmop{supp}\big(\tau_{-\pi/2}T\big)\subset P .
		\end{equation}
		Here $\tmop{supp}T\subset N$, with $\tmop{supp}T$ closed and $N$ open, means that $T$ vanishes on an open set containing the closed arcs $\T\setminus N$. This is where the exclusion of the punctured axes from the spectral sense enters: a point mass of $T$ at an axis direction would put a whole punctured axis into the spectrum.

		\emph{Step 4: necessity, the odd part.} Assume $\tmop{supp}\hat G\subset Q_{2,4}\cup\{0\}$. The reflection $z\mapsto-z$ replaces $\Omega(\theta)$ by $\Omega(\theta+\pi)$ and $\hat G(\xi)$ by $\hat G(-\xi)$; as $Q_{2,4}\cup\{0\}$ is symmetric, the transforms of $G_o=\Omega_o(\theta)$ and $G_e=\Omega_e(\theta)$ are both supported in $Q_{2,4}\cup\{0\}$. By \eqref{Ghat}, \eqref{TOmega} and \eqref{crit} applied to $\Omega_o$ (for which $\Omega_0=0$), $\tmop{supp}\Omega_o'\subset P$: the distribution $\Omega_o'$ vanishes on an open set containing the two closed arcs of non-positive slope, i.e.\ $\Omega_o$ is locally constant on a neighborhood of these arcs.

		\emph{Step 5: necessity, the even part.} By \eqref{Ghat}, \eqref{TOmega} and \eqref{crit} applied to $\Omega_e$ (the term at the origin in \eqref{Ghat} is irrelevant for \eqref{crit}), $\tmop{supp}(|D|\Omega_e)\subset P$; in particular $|D|\Omega_e=(\mathcal H\Omega_e)'$ vanishes on the open arc $I=(\pi/2,\pi)$, so the conjugate function $h=\mathcal H\Omega_e\in L^2(\T)$ equals a constant $\kappa$ a.e.\ on $I$. Consider
		$$g(z)=\Omega_{e,0}+2\sum_{k\geq1}\Omega_{e,k}\,z^k\in H^2(\mathbb D),$$
		whose boundary function is $\Omega_e+ih$ (recall that $\Omega_e$ is real, so $\Omega_{e,-k}=\overline{\Omega_{e,k}}$); thus $g=u+iv$ with $u=P[\Omega_e]$ and $v=P[h]$ the Poisson integrals of the boundary functions. Since $h-\kappa=0$ a.e.\ on $I$,
		$$v(z)=\kappa+\frac1{2\pi}\int_{\T\setminus I}\Re\frac{w+z}{w-z}\,\big(h(w)-\kappa\big)\,|dw| ,$$
		and the right-hand side is harmonic on $\C\setminus(\T\setminus I)$, an open set containing $\mathbb D\cup I$, and equals $\kappa$ at every point of $I$, because the Poisson kernel $\Re\frac{w+z}{w-z}=\frac{1-|z|^2}{|w-z|^2}$ vanishes on the unit circle away from its pole. Let $z_0\in I$ and let $B$ be a disc centered at $z_0$ and contained in $\C\setminus(\T\setminus I)$. On $B$ the harmonic function $v$ has a harmonic conjugate: there is $U$ with $U+iv$ holomorphic on $B$. On the convex set $B\cap\mathbb D$ the function $g-(U+iv)$ is holomorphic with vanishing imaginary part, hence a real constant, which we absorb into $U$. Thus $g$ extends holomorphically to $\mathbb D\cup B$, and doing this at every point of $I$ we obtain a holomorphic extension $\tilde g$ of $g$ to a neighborhood of $\mathbb D\cup I$ (the local extensions agree on overlaps by the identity theorem), with $\Im\tilde g=\kappa$ on $I$. The radial limits of $g$, which exist a.e.\ and equal $\Omega_e+ih$, coincide on $I$ with the values of the continuous function $\tilde g$; hence $\Omega_e=\Re\tilde g$ a.e.\ on $I$. Now $\Re\tilde g$ is real-analytic on $I$ and a.e.\ equal to a function with values in $\{-1,0,1\}$; the open subset of $I$ where it takes other values is null, hence empty, and by connectedness $\Re\tilde g\equiv\kappa_1\in\{-1,0,1\}$ on $I$. So $\tilde g=\kappa_1+i\kappa$ on the arc $I$, and by the identity theorem $g\equiv\kappa_1+i\kappa$; in particular $\Omega_e=\kappa_1$ a.e.\ on $\T$.

		If $\kappa_1=\pm1$, then $\Omega(\theta)+\Omega(\theta+\pi)=\pm2$ a.e., which forces $\Omega\equiv\pm1$, i.e.\ $A$ null or conull, contrary to the hypothesis. Hence $\kappa_1=0$ and $\Omega=\Omega_o$ is odd, which is (a); Step 4 then gives (b).

		\emph{Step 6: sufficiency and the formula.} Conversely, let (a) and (b) hold. Then $\Omega_e=0$, $\Omega_0=0$, and \eqref{Ghat}, \eqref{TOmega} read
		$$\hat G=-ic\,\big(\tau_{\pi/2}\Omega'\big)\otimes\rho^{-2},$$
		which is the precise meaning of the expression $|\xi|^{-2}\,\Omega'(\varphi-\pi/2)$ in the statement. By (b), $\tmop{supp}\Omega'\subset P$, so \eqref{crit} gives $\tmop{supp}\hat G\subset Q_{2,4}\cup\{0\}$, and by \eqref{supp} the spectrum consists of the origin and the lines through it in the directions $\varphi\in\tmop{supp}\Omega'+\frac\pi2$, the normals to the rays $\theta\in\tmop{supp}\Omega'$ on which $\Omega$ is not locally constant. Finally, (a) determines $\Omega$ on one arc of positive slope from the other, and (a) and (b) together determine $\Omega$ on the arcs of non-positive slope from its values near their endpoints, which lie in the arcs of positive slope; so $A$ is determined by its trace on one open arc of positive slope, and this trace is an arbitrary measurable set subject only to (b) near the endpoints. The set $F$ is a fan exactly when $\Omega$ has finitely many jumps, i.e.\ when the trace is a finite union of intervals.
	\end{proof}

	Two consistency checks. For $\Omega=\tmop{sign}\cos\theta$, i.e.\ $G=\tmop{sign}x$, one has $\Omega'=2\delta_{-\pi/2}-2\delta_{\pi/2}$, so \eqref{Ghat} gives $\hat G$ proportional to $(\delta_0-\delta_\pi)\otimes\rho^{-2}=\pv(1/\xi_1)\otimes\delta(\xi_2)$, as it should be. For the even ``bowtie'' $\Omega=\tmop{sign}\sin2\theta$, i.e.\ $G=\tmop{sign}x\,\tmop{sign}y$, the distribution $|D|\Omega=\mathcal H(\Omega')$ is the conjugate of a comb of four point masses, a multiple of $\pv\,1/\sin2\varphi$, and \eqref{Ghat} gives $\hat G$ proportional to $\pv\,1/(\xi_1\xi_2)$, the product of the one-dimensional transforms, supported in the whole plane.

	Theorem \ref{Tcone} is the homogeneous counterpart of Theorem \ref{Tinterval}: there the invariant sets with interval slices were parametrized by singular Herglotz functions, here the invariant cones are parametrized by arbitrary measurable subsets of an interval of directions. In the wave language of Section \ref{sec:waves} the cones are the self-similar fields $F(x,t)=\Omega(\arg(x,t))$ -- solutions of the linearized Riemann problem for a continuum of characteristic speeds -- and the spectrum, a cone of lines, is two-dimensional as soon as the jump set of $\Omega$ has positive measure.

	\subsection{Rigidity of multi-directional examples}

	How special are these examples? Write a bounded candidate whose spectrum is carried by finitely many lines through the origin as
	$$2\chi_S-1=\sum_{d=1}^Dg_d(\ell_d),$$
	a finite sum of bounded waves along pairwise non-parallel linear functionals $\ell_d=a_dy-b_dx$ with $b_d/a_d>0$ (so that all frequencies lie in $Q_{2,4}$; this form of $2\chi_S-1$ is forced by the spectral support, and for lattice-periodic $S$ the profiles $g_d$ are periodic). The fans of Proposition \ref{Pfan} show that with step profiles every odd $D$ occurs, all lines then passing through one point. The following statements concern the complementary situation of profiles which are not integer-valued -- sawtooths rather than steps -- and indicate that there $D\in\{1,3\}$, with $D=1$ the unions of parallel strips and $D=3$ the carry example.

	\begin{lemma}[two directions are impossible]\label{Ltwo}
		If $g_1(u)+g_2(v)$ takes only two values on a set of full measure, with $u,v$ independent variables, then $g_1$ or $g_2$ is a.e.\ constant.
	\end{lemma}

	\begin{proof}
		Suppose $g_1$ takes values $a_1\neq a_2$ and $g_2$ values $b_1\neq b_2$, each on positive measure. The four sums $a_i+b_j$ must lie in a two-element set. From $a_1+b_1\neq a_1+b_2$ the two values are $v_1=a_1+b_1$ and $v_1+\delta$, $\delta=b_2-b_1$; then $a_2+b_1\in\{v_1,v_1+\delta\}$ forces $a_2-a_1=\delta$, and $a_2+b_2=v_1+2\delta$ lies outside the set.
	\end{proof}

	\begin{proposition}[character rigidity]\label{Pchar}
		Suppose $\sigma=g_1(u)+g_2(v)+g_3(u+v)$ is $\{0,1\}$-valued a.e., with bounded measurable $g_d$, none a.e.\ constant. Then there exist $\mu\in\R$, constants $\gamma_d$, and integer-valued measurable functions $N_d$ such that
		$$g_1(x)=\mu x-N_1(x)+\gamma_1,\quad g_2(x)=\mu x-N_2(x)+\gamma_2,\quad g_3(x)=-\mu x-N_3(x)+\gamma_3 ,$$
		so that $N_1(u)+N_2(v)+N_3(u+v)$ is two-valued. Here $\mu=0$ exactly when the $g_d$ are integer-valued up to additive constants. When $\mu\neq0$ the three frequencies resonate -- the third functional is the sum of the first two -- which is why the third family of Theorem \ref{Tcarry}, the sum of the like-oriented functionals $y-x$ and $y-\frac x2$, has the intermediate slope $\frac 34$. The case $\mu=0$ occurs for the fans of Proposition \ref{Pfan}, where the resonance takes the form of the concurrency of the three lines.
	\end{proposition}

	\begin{proof}[Proof (sketch)]
		Since $\sigma\in\Z$ a.e., $e^{2\pi ig_1(u)}e^{2\pi ig_2(v)}e^{2\pi ig_3(u+v)}=1$ a.e. Setting $h=e^{-2\pi ig_3}$ we get $h(u+v)=f_1(u)f_2(v)$ with unimodular measurable $f_j$; the multiplicative Cauchy functional equation for measurable functions forces $h$ and $f_j$ to be unimodular characters times constants, i.e.\ $e^{2\pi ig_d(x)}=c_de^{2\pi i\mu_dx}$ with $\mu_1=\mu_2=-\mu_3=\mu$ and $c_1c_2c_3=1$. Taking logarithms yields the stated form; $\mu=0$ means that each $e^{2\pi ig_d}$ is constant, i.e.\ that $g_d$ is integer-valued up to a constant.
	\end{proof}

	Proposition \ref{Pchar} leaves two cases. If $\mu\neq0$, boundedness of the $g_d$ says that $N_1-\mu x$, $N_2-\mu x$ and $N_3+\mu x$ are bounded, and one expects the $N_d$ to be floor functions. We record the statement we need as a conjecture; it belongs to the circle of ideas around Beatty sequences and bounded remainder sets, cf.\ \cite{Gr}.

	\begin{conjecture}[floor rigidity]\label{Lsturm}
		Let $\mu\neq0$ and let $N_1,N_2,N_3:\R\to\Z$ be measurable, with $N_1-\mu x$, $N_2-\mu x$ and $N_3+\mu x$ bounded, such that $N_1(u)+N_2(v)+N_3(u+v)$ takes only two values a.e. Then, up to additive integer constants, $N_1(x)=\lfloor\mu x+\rho_1\rfloor$, $N_2(x)=\lfloor\mu x+\rho_2\rfloor$ and $N_3(x)=-\lfloor\mu x+\rho_1+\rho_2\rfloor$ for some $\rho_1,\rho_2$.
	\end{conjecture}

	Granting Conjecture \ref{Lsturm}, every $D=3$ example with $\mu\neq0$ is an affine image of the carry set of Theorem \ref{Tcarry} or of its complement. If $\mu=0$ the profiles are integer-valued up to constants, and the three-line fans of Proposition \ref{Pfan} are solutions, with $N_1=H(\cdot-a)$, $N_2=H(\cdot-b)$, $N_3=-H(\cdot-a-b)$; we conjecture that, up to constants and complementation, these are the only bounded integer-valued solutions of the three-term equation. Note that the analogue of Conjecture \ref{Lsturm} fails for $\mu=0$: the Heaviside triple is a bounded integer-valued solution which is not of floor type. For more directions the two regimes separate completely: with integer-valued profiles the fans give solutions for every odd $D$, while for sawtooth-type profiles we have the following obstruction.

	\begin{proposition}[four waves, generic resonance]\label{Pfour}
		Let $z_d=\mu_d\ell_d+\rho_d$, $d=1,\dots,4$, with all $\mu_d\neq 0$, the functionals $\mu_d\ell_d$ pairwise non-proportional and satisfying the single resonance $\sum\mu_d\ell_d=0$, and no other rational relation. Then $\sum_{d=1}^4\{z_d\}$ takes at least three values on sets of positive measure; in particular it is not two-valued.
	\end{proposition}

	\begin{proof}
		By Weyl equidistribution, the essential range of $(\{z_1\},\dots,\{z_4\})$ is the subtorus $\{\theta\in(\R/\Z)^4:\ \theta_1+\theta_2+\theta_3+\theta_4\equiv c\}$ for the appropriate $c$. On it, $\sum\theta_d\in\{\tilde c,\tilde c+1,\tilde c+2,\tilde c+3\}$ where $\tilde c=\{c\}$; choosing all coordinates small, one coordinate close to $1$, and all coordinates close to $1$, respectively, realizes at least three distinct levels on relatively open sets (with obvious modifications when $\tilde c=0$).
	\end{proof}

	\subsection{Sets with boundary on countably many lines}

	\begin{theorem}[finite cross-sections]\label{Tlines}
		Let $S$ have all cross-sections of finite measure, horizontal ones uniformly bounded, let $\H\chi_S=\chi_S$, and suppose $\partial S$ is contained in a locally finite countable union of lines with finitely many directions. Then $S$ is, up to a null set, a countable union of parallel strips of one positive slope; that is, $S$ is a diagonal set $\{\alpha y-\beta x\in E\}$ with $E$ a countable union of intervals of finite total length.
	\end{theorem}

	\begin{proof}[Proof (sketch)]
		On any interval of heights free of intersection points of the boundary lines, the slice of $S$ is a locally finite union of intervals with affine endpoints, of total length at most $M$, and
		$$F(t,y)=\sum_j\frac{e^{-it\alpha_j(y)}-e^{-it\beta_j(y)}}{it},\qquad \alpha_j,\beta_j\ \text{affine}.$$
		Only finitely many of these intervals have non-parallel edges (their lengths are $|a_j-a_j'|\,|y-y_j|$ with $|a_j-a_j'|$ bounded below, since there are finitely many slopes, and with $y_j$ outside the interval of heights, and the lengths add up to at most $M$), and the remaining terms, with parallel edges of slope $a$, are bounded by $\ell_j e^{t|a||\Im y|}$ for complex $y$; hence the sum converges locally uniformly in $y\in\C$ and defines an entire function of $y$. By Lemma \ref{Lslice}, $F(t,\cdot)$ is the boundary value of a single function analytic in $\C_-$; hence the local formulas on the two sides of any transition height must agree identically. For fixed slope, the coefficients $e^{-itb}$ attached to distinct intercepts $b$ are linearly independent as functions of $t$, so the multiset of endpoint lines (slope, intercept, orientation) cannot change at any height: there are no corners, no appearing or disappearing intervals, no bowties. Consequently every slice endpoint moves along one full line for all $y$; the two edges of one interval must be parallel (otherwise they cross, a forbidden transition), and any two strips must be parallel (otherwise their edges cross). Two non-parallel strips also cannot coexist for a direct reason: their union differs from the sum of their indicators by the indicator of a bounded parallelogram $P$, and $(\H-\mathcal I)\chi_P\neq 0$ since $\widehat{\chi_P}$ is entire and cannot vanish on an open quadrant (the uncertainty argument of Section 2). Finally the common slope is positive, since the vertical wave factor $e^{-it\sigma y}$ must be inner in $\C_-$ for $t>0$; horizontal edges are excluded as transitions, and vertical strips have infinite vertical cross-sections.
	\end{proof}

	We expect the conclusion to hold without the restriction to finitely many directions.

	Together with the results of the previous subsection this supports the following statement, which is consistent with our computations in cyclic models and which we expect to be the complete answer.

	\begin{conjecture}[trichotomy]\label{Cdich}
		Let $S\subset\R^2$ be measurable, with $\partial S$ contained in a locally finite countable union of lines, and suppose $\H(2\chi_S-1)=2\chi_S-1$ in the spectral sense (equivalently, $\H\chi_S=\chi_S$ when the cross-sections have finite measure). Then, up to null sets and complementation, $S$ is a union of parallel strips of a single positive slope, or a fan of Proposition \ref{Pfan} (an odd number of concurrent lines, alternate sectors), or an affine image of the carry set of Theorem \ref{Tcarry}.
	\end{conjecture}

	The combinatorial half of the conjecture is settled in Section \ref{sec:SG}: by Theorem \ref{Tarr} and Corollary \ref{Carr}, the boundary arrangement of such a set (with the wave decomposition of Proposition \ref{Pchar}) is necessarily a family of parallel lines, a pencil of concurrent lines, or an affine image of the triangular arrangement of the carry set, and in the concurrent case $S$ is a fan. What remains open is the arithmetic statement that, on the triangular arrangement, the profiles must be the carry sawtooths.

	\begin{remark}[computational evidence]\label{Rcomp}
		In the cyclic model, the $\{0,1\}$-valued functions on $\Z_N^2$ of the form $A(i)+B(j)+C(i+j)$, with arbitrary real $A,B,C$, can be enumerated exhaustively. For $N=3$ all $56$ such patterns, and for $N=4$ all $140$, are accounted for by: carry patterns (with their translates, unit multiples and complements), one-directional patterns (discrete unions of strips), and, for $N=4$, hybrid patterns which are disjoint sums of a carry pattern and strips threaded through fibers where a fractional part vanishes \emph{exactly}. The hybrids are artifacts of the lattice: in the continuum the corresponding fibers have measure zero, and an equidistribution argument shows that no strip family can be disjoint from a carry set. Fans are not periodic and therefore do not appear in cyclic models. Scans over all small-integer resonant configurations of $4$ and $5$ directions with sawtooth profiles never produce a two-valued sum, in agreement with Proposition \ref{Pfour}; the fans, being finite sums of single-direction waves, are invariant by linearity alone, and only their two-valuedness needs checking. The unproved ingredients in the classification are Conjecture \ref{Lsturm} for $\mu\neq0$ and the determination of all bounded integer-valued solutions for $\mu=0$, where the fans occur; in the latter case Theorem \ref{Tarr} shows that a locally finite arrangement carrying step profiles is parallel, concurrent, or an affine triangular lattice, and the conjecture is that the third case does not occur with integer-valued profiles.
	\end{remark}

	\section{A combinatorial coda: Sylvester--Gallai duality}\label{sec:SG}

	The line-boundary classification has a purely combinatorial shadow, which connects it to configuration problems of Sylvester--Gallai type and, somewhat unexpectedly, to a recent example of B\'ar\'any, Du, Schwarz, Yuan and Zamfirescu \cite{BDSYZ}, which turns out to be precisely the projective dual of an affine copy of the triangular lattice (Proposition \ref{Pdual}). It also contains a complete classification of the locally finite line arrangements without ordinary points (Theorem \ref{Tarr} below), which settles the combinatorial half of Conjecture \ref{Cdich}.

	\begin{lemma}[local parity]\label{Lparity}
		Let $S$ be as in Conjecture \ref{Cdich}, with wave decomposition $2\chi_S-1=\sum g_d(\ell_d)$ as in Proposition \ref{Pchar}, so that each boundary family carries a sawtooth or a step whose jumps have a fixed orientation $\varepsilon_d=\pm 1$. If $k$ boundary lines meet at a point, then $k$ is odd, and the orientations $\varepsilon_d$ alternate with respect to the angular order of the slopes.
	\end{lemma}

	\begin{proof}
		Orient the functionals alike, replacing $g_d$ by $g_d(-\,\cdot)$ where necessary, and walk around a small circle centered at the intersection point. Since all slopes are positive, the gradients $\nabla\ell_d$ lie in one open quadrant, and the $2k$ crossings occur in angular order as all $k$ lines crossed in the increasing direction of their functionals, followed by the same $k$ lines crossed in the decreasing direction. The indicator changes by $-\varepsilon_d$ at an increasing crossing of family $d$ and by $+\varepsilon_d$ at a decreasing one. A two-valued function must alternate at consecutive jumps, so in the cyclic sequence $(-\varepsilon_1,\dots,-\varepsilon_k,\varepsilon_1,\dots,\varepsilon_k)$ consecutive entries have opposite signs: within a block, $\varepsilon_{i+1}=-\varepsilon_i$; at the junction, $\varepsilon_1=\varepsilon_k$. The two conditions force $(-1)^{k-1}=1$, i.e.\ $k$ odd, and the alternation of orientations.
	\end{proof}

	\begin{corollary}\label{Cordinary}
		No intersection point of the boundary arrangement lies on exactly two lines (``there are no ordinary points''). At a triple point the middle slope carries the opposite orientation; this is consistent with, and in fact predicts, the resonance of Proposition \ref{Pchar}: once the three functionals are oriented alike (which reflecting the profiles always arranges), the resonant one is a positive combination of the other two, and its level lines have the intermediate slope. The fans of Proposition \ref{Pfan} realize the minimal configurations: a single point of multiplicity $k$, for every odd $k$. See Figure \ref{fig:triple}.
	\end{corollary}

	\begin{figure}[htb]
		\centering
		\includegraphics[width=.6\textwidth]{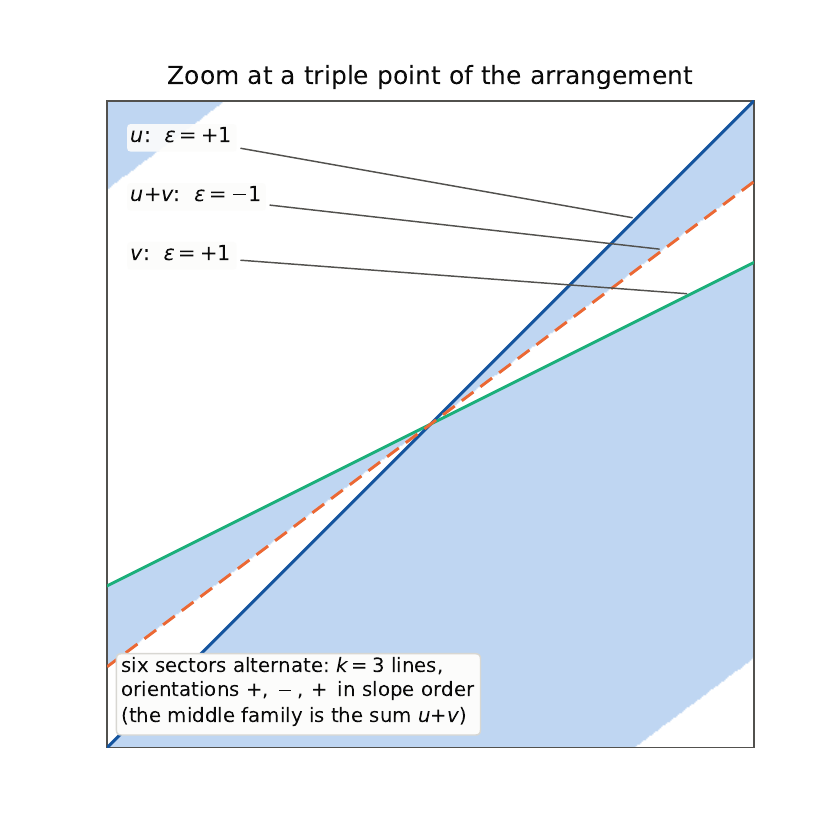}
		\caption{The set of Theorem \ref{Tcarry} near a triple point of its boundary arrangement: six sectors with alternating membership, $k=3$ lines, jump orientations $+,-,+$ in slope order, the middle family (the resonant sum $u+v$) inverted -- the local picture of Lemma \ref{Lparity} and, globally, the three-line fan of Proposition \ref{Pfan}.}
		\label{fig:triple}
	\end{figure}

	\subsection{Locally finite arrangements without ordinary points}

	For locally finite arrangements, the absence of ordinary points is not merely a necessary condition: it determines the arrangement completely. Throughout this subsection, $\mathcal A$ is a family of distinct lines in the plane, \emph{locally finite} in the sense that every disc meets finitely many lines of $\mathcal A$; a \emph{vertex} of $\mathcal A$ is a point lying on at least two of its lines, an \emph{ordinary} vertex lies on exactly two, and the \emph{faces} of $\mathcal A$ are the connected components of the complement of $\bigcup\mathcal A$, which are open convex polygonal regions. For the history of Sylvester's problem and its relatives see \cite{BM}.

	\begin{theorem}[arrangements without ordinary points]\label{Tarr}
		Let $\mathcal A$ be a locally finite family of at least two distinct lines in the plane with no ordinary vertex. Then exactly one of the following holds:

		\emph{(i)} the lines of $\mathcal A$ are all parallel;

		\emph{(ii)} the lines of $\mathcal A$ all pass through one point;

		\emph{(iii)} some affine map of the plane takes $\mathcal A$ onto the triangular arrangement
		$$\mathcal A_\triangle=\{x=m\}\cup\{y=n\}\cup\{x+y=k\},\qquad m,n,k\in\Z .$$

		\noindent Conversely, each of \emph{(i)--(iii)} has no ordinary vertex. In case \emph{(iii)} every vertex of $\mathcal A$ lies on exactly three lines and every face of $\mathcal A$ is a triangle.
	\end{theorem}

	Case (ii) cannot be dispensed with -- the arrangements of the fans of Proposition \ref{Pfan} satisfy the hypotheses -- and this is why local finiteness and the absence of ordinary vertices alone do not force triangular faces. The converse direction is trivial for (i) and (ii) and is the computation of Remark \ref{Rgeom} for (iii): in $\mathcal A_\triangle$, the lines $x=m$ and $y=n$ meet on $x+y=m+n$, the lines $x=m$ and $x+y=k$ meet on $y=k-m$, and the lines $y=n$ and $x+y=k$ meet on $x=k-n$. Note also that in cases (i) and (ii) the arrangement may be finite, while in case (iii) it is necessarily infinite -- in accordance with the Sylvester--Gallai theorem, which forbids finite examples of this kind. The theorem thus complements Sylvester--Gallai in the infinite, locally finite setting; in dual form it also identifies the bounded counterexample of \cite{BDSYZ}, Theorem 2.2, as essentially the only one of its kind (Proposition \ref{Pdual} and Corollary \ref{Cdual} below), while the non-locally-finite setting -- dense arrangements such as the dual of $\mathbb Q^2$ -- is genuinely different.

	For the proof, assume that $\mathcal A$ is locally finite with no ordinary vertex and that its lines are neither all parallel nor all concurrent; we must produce the affine map of (iii). Call a closed triangle \emph{line-bounded} if its three sides lie on three lines of $\mathcal A$.

	\begin{lemma}[a triangle exists]\label{Lextri}
		Some three lines of $\mathcal A$ bound a triangle.
	\end{lemma}

	\begin{proof}
		Choose non-parallel $m_1,m_2\in\mathcal A$ and let $P=m_1\cap m_2$. Since not all lines pass through $P$, there is $m_3\in\mathcal A$ with $P\notin m_3$. If $m_3$ is parallel to neither $m_1$ nor $m_2$, then $m_1,m_2,m_3$ are pairwise non-parallel and not concurrent (a common point would be $P\in m_3$), so they bound a triangle. Otherwise, say $m_3\parallel m_1$; then $Q=m_3\cap m_2$ is a vertex, $Q\neq P$, so a third line $m_4\in\mathcal A$ passes through $Q$, $m_4\notin\{m_2,m_3\}$. We cannot have $m_4\parallel m_1$, since $m_3$ is the line through $Q$ parallel to $m_1$. Hence $m_1,m_2,m_4$ are pairwise non-parallel and not concurrent: a common point would be $P$, and $P,Q\in m_4$ would give $m_4=m_2$. So $m_1,m_2,m_4$ bound a triangle.
	\end{proof}

	\begin{lemma}[counting]\label{Lcount}
		Let $T$ be a line-bounded triangle, and let $G$ be the planar graph cut out on $T$ by the lines of $\mathcal A$: its vertices are the vertices of $\mathcal A$ interior to $T$, the points where lines of $\mathcal A$ cross $\partial T$, and the three corners; its edges are the segments of lines of $\mathcal A$ and of $\partial T$ between consecutive vertices; its faces are the components of the interior of $T$ minus the lines. Let $k$ and $l$ be the numbers of interior and of non-corner boundary vertices, and $E$ the number of edges. Then
		$$E=3k+2l+3,$$
		and consequently: \emph{(a)} every face of $G$ is a triangle; \emph{(b)} every interior vertex lies on exactly three lines of $\mathcal A$; \emph{(c)} every non-corner boundary vertex lies on exactly three lines of $\mathcal A$, the side and two lines crossing into $T$; \emph{(d)} no line of $\mathcal A$ other than the two sides enters the interior of $T$ at a corner.
	\end{lemma}

	\begin{proof}
		If $V$ is a non-corner boundary vertex, every line of $\mathcal A$ through $V$ other than the side enters the interior of $T$: it crosses the side transversally at $V$, and a half-neighborhood of $V$ on the inner side lies in $T$. Since $V$ lies on at least three lines, at least two of them enter, so $\deg V\geq4$ (two boundary edges and at least two interior ones). Interior vertices lie on at least three lines, each contributing two edges, so their degree is at least $6$; corners have degree at least $2$. With $V=k+l+3$ vertices in total, summing the degrees gives
		\begin{equation}\label{arrlower}
			2E\geq 6k+4l+6 .
		\end{equation}
		The graph is connected: every edge lies on a segment of a line inside $T$ whose endpoints lie on the connected set $\partial T$. Since $T$ is a disc, Euler's formula gives $V-E+F=1$ for the number $F$ of faces, so $F=E-k-l-2$. Interior edges border two faces, the $l+3$ boundary edges border one, so the face degrees add up to $2E-(l+3)$; every face is a convex polygon with at least three boundary edges, whence
		\begin{equation}\label{arrupper}
			3(E-k-l-2)\leq 2E-(l+3),\qquad\text{i.e.}\qquad E\leq 3k+2l+3 .
		\end{equation}
		Comparing \eqref{arrlower} and \eqref{arrupper} gives $E=3k+2l+3$, and every inequality used must be an equality: equality in \eqref{arrupper} is (a), and equality in \eqref{arrlower} gives degrees exactly $6$, $4$, $2$, which is (b), (c), (d).
	\end{proof}

	Applying Lemma \ref{Lcount} to the triangle of Lemma \ref{Lextri}, we see that $\mathcal A$ has at least one triangular face.

	\begin{lemma}[sides of a triangular face]\label{Lsides}
		Let $F$ be a triangular face of $\mathcal A$ with closure the triangle $ABC$. Then $A,B,C$ are the only vertices of $\mathcal A$ on $\partial F$; each side of $ABC$ is a single edge of the arrangement, and the face adjacent to $F$ along a side shares that entire side.
	\end{lemma}

	\begin{proof}
		If a vertex $V$ lay in the relative interior of the side $BC$, a line of $\mathcal A$ through $V$ other than the line $BC$ would cross it transversally at $V$ and enter $F$, which no line may do. Hence the sides are edges. If $F'$ is the face adjacent to $F$ along $BC$, the side of $F'$ containing the edge $BC$ cannot be longer, since then $B$ or $C$ would be a vertex interior to a side of $F'$, impossible by the same argument.
	\end{proof}

	\begin{lemma}[the third line]\label{Lthird}
		Let $F$ be a triangular face with vertices $A,B,C$. Then every line of $\mathcal A$ through $B$ other than the lines $AB$ and $BC$ is parallel to $AC$; consequently exactly three lines of $\mathcal A$ pass through $B$, namely $AB$, $BC$ and one line $\ell_B\parallel AC$.
	\end{lemma}

	\begin{proof}
		Let $\ell\in\mathcal A$ pass through $B$, $\ell\neq AB, BC$, and suppose $\ell\not\parallel AC$; let $D$ be the intersection of $\ell$ with the line $AC$. If $D$ lay in the open segment $(A,C)$, the segment $BD$ would cross the interior of the face $F$; and $D=A$ or $D=C$ would force $\ell=AB$ or $\ell=BC$. Hence $D$ lies outside $[A,C]$, and, interchanging the labels $A$ and $C$ if necessary, we may assume that $A$ lies strictly between $D$ and $C$. The triangle $T'$ with corners $D,B,C$ is line-bounded, with sides on $\ell$, $AC$ and $BC$. The upper bound \eqref{arrupper}, which used only Euler's formula and convexity of the faces, gives $E'\leq 3k'+2l'+3$ for its subdivision. In the lower bound, the corner $B$ now supports at least \emph{three} edges: the boundary edges toward $D$ and toward $C$, and the edge along the segment $BA$, which lies inside $T'$ because $A$ is an interior point of the side $DC$. Hence
		$$2E'\geq 6k'+4l'+(2+2+3)=6k'+4l'+7,$$
		so $E'\geq 3k'+2l'+\tfrac72$, a contradiction. Therefore $\ell\parallel AC$. At most one line through $B$ is parallel to $AC$, and $B$ lies on at least three lines; so exactly three.
	\end{proof}

	\begin{lemma}[reflected neighbor]\label{Lrefl}
		Let $F$ be a triangular face with vertices $A,B,C$, and let $\ell_B\parallel AC$, $\ell_C\parallel AB$ be the third lines through $B$ and $C$. Then the face adjacent to $F$ along $BC$ is the open triangle with vertices $B$, $C$ and $A'=\ell_B\cap\ell_C$, the reflection of $A$ in the midpoint of $BC$. In particular it is a triangular face with sides parallel to those of $F$.
	\end{lemma}

	\begin{proof}
		The point $A'=B+C-A$ lies on the opposite side of the line $BC$ from $A$, so the triangle $T'$ with corners $B,C,A'$ is line-bounded (sides on $BC$, $\ell_B$, $\ell_C$) and lies on that side. Let $F'$ be the face adjacent to $F$ along the edge $BC$. Near an interior point of $BC$, a half-disc on the far side lies in $T'$, so $F'$ meets the interior of $T'$; since $\partial T'$ lies on lines of $\mathcal A$, in fact $F'\subset T'$. By Lemma \ref{Lcount} applied to $T'$, $F'$ is a triangle, and by Lemma \ref{Lsides} its closure contains the segment $BC$ with $B$ and $C$ as corners. The remaining two sides of $F'$ lie on lines through $B$ and through $C$ other than $BC$, that is, on one of $AB,\ell_B$ and one of $AC,\ell_C$ (Lemma \ref{Lthird}), and the third corner is their intersection, on the far side of $BC$. Now $AB\cap AC=A$ lies on the near side, while $AB\parallel\ell_C$ and $\ell_B\parallel AC$ do not intersect; the only possibility is $\ell_B\cap\ell_C=A'$.
	\end{proof}

	\begin{proof}[Proof of Theorem \ref{Tarr}]
		Let $F_0$ be a triangular face and $\alpha,\beta,\gamma$ the directions of its sides. Any face $F_*$ can be joined to $F_0$ by a segment avoiding the locally finite set of vertices; the segment crosses finitely many lines, so there is a finite chain of faces from $F_0$ to $F_*$ in which consecutive faces are adjacent along edges. By Lemma \ref{Lrefl} and induction along the chain, every face of $\mathcal A$ is a triangle with sides of directions $\alpha,\beta,\gamma$. If $\ell\in\mathcal A$ and $x\in\ell$ is not a vertex, then $x$ lies on a common boundary edge of two faces, and this edge lies on $\ell$; hence every line of $\mathcal A$ has one of the three directions, and $\mathcal A$ splits into three nonempty families of parallel lines. By Lemma \ref{Lthird}, every vertex lies on exactly three lines, one from each family.

		Applying an affine map, we may assume the three directions are those of $x=0$, $y=0$ and $x+y=0$, so that
		$$\mathcal A=\{x=a:\ a\in A\}\cup\{y=b:\ b\in B\}\cup\{x+y=c:\ c\in C\}$$
		with locally finite $A,B,C\subset\R$. The absence of ordinary vertices says precisely that
		\begin{equation}\label{arrsum}
			A+B\subset C,\qquad C-A\subset B,\qquad C-B\subset A
		\end{equation}
		(the lines $x=a$ and $y=b$ meet at $(a,b)$, which must lie on a line of the third family, so $a+b\in C$, and similarly for the other two pairs). Each family contains at least two lines: if, say, $B=\{b\}$ while $a\neq a'$ belong to $A$, then $a'+b\in C$ and the lines $x=a$, $x+y=a'+b$ meet at $(a,\,a'+b-a)$, forcing $a'+b-a\in B$, i.e.\ $a'=a$; and if two families were singletons, $\mathcal A$ would consist of three pairwise crossing lines, which are concurrent or have an ordinary vertex. Now fix $b_1\neq b_2$ in $B$ and $a_1\in A$. By \eqref{arrsum}, $A+b_i\subset C$ and $C-b_i\subset A$, so $A=C-b_i$ for $i=1,2$; hence $C=C+(b_2-b_1)$, and symmetrically $C=C+(a_2-a_1)$ for all $a_1,a_2\in A$. Thus $C$ is invariant under the group $G$ generated by $(A-A)\cup(B-B)$, and, being nonempty and locally finite, $C$ can only be invariant under a discrete group: $G=\delta\Z$ with $\delta>0$. From $A=C-b_1$ and $B=C-a_1$ we get $A-A=B-B=C-C\subset G$, and together with $C+G=C$ this gives $C=c_0+\delta\Z$ for any $c_0\in C$; hence also $A=a_0+\delta\Z$ and $B=b_0+\delta\Z$, with $a_0+b_0\in C$. The affine map $(x,y)\mapsto((x-a_0)/\delta,\,(y-b_0)/\delta)$ takes $\mathcal A$ onto $\mathcal A_\triangle$.
	\end{proof}

	Note that the proof yields slightly more than stated: in case (iii) the three intercept sets are forced to be arithmetic progressions with a common spacing, so there are no arrangements of type (iii) with irregular spacings. Combining the theorem with Lemma \ref{Lparity} we obtain the combinatorial half of the trichotomy.

	\begin{corollary}\label{Carr}
		Let $S$ be as in Lemma \ref{Lparity}: $\partial S$ is contained in a locally finite union of lines and $2\chi_S-1=\sum g_d(\ell_d)$ as in Proposition \ref{Pchar}, with nonconstant profiles carrying jumps of fixed orientations. Then the arrangement of the boundary lines is a family of parallel lines, a pencil of concurrent lines, or an affine image of the triangular arrangement $\mathcal A_\triangle$. In the first case $S$ is a diagonal set; in the second, $S$ is a fan of Proposition \ref{Pfan}; in the third, after the affine normalization, the boundary lines are $u\in\Z$, $v\in\Z$, $u+v\in\Z$, the arrangement of the carry set. Conjecture \ref{Cdich} is thereby reduced to its arithmetic core: to show that the only two-valued configurations $g_1(u)+g_2(v)+g_3(u+v)$ with bounded profiles whose jumps lie exactly on $\Z$ are, up to constants and complementation, the carry configuration of Theorem \ref{Tcarry}.
	\end{corollary}

	\begin{proof}
		By Corollary \ref{Cordinary} the boundary arrangement has no ordinary vertices, and it is locally finite, so Theorem \ref{Tarr} applies (the cases of at most one boundary line fall under (i)). In case (i), $S$ is a union of faces of a family of parallel lines of one positive slope, i.e.\ a diagonal set. In case (ii), $S$ is a union of sectors at the common point $p$; each boundary line carries a genuine jump, so $\chi_S$ changes across every ray, and around $p$ the values $0$ and $1$ alternate; by Lemma \ref{Lparity} the number of lines is odd, and $S$ is an alternating union of sectors of concurrent lines of positive slopes, i.e.\ a fan (Proposition \ref{Pfan}(i)). Case (iii) is Theorem \ref{Tarr}, together with the observation that the three families inherit the positive slopes of the wave decomposition.
	\end{proof}

	Projective duality takes the line $ax+by+c=0$ to the point with homogeneous coordinates $[a:b:c]$ in the dual projective plane; concurrent lines dualize to collinear points, and pencils of parallel lines dualize to collinear points as well, carried by lines through the dual point of the line at infinity. Under this duality the boundary arrangement becomes a countable point set $S^*$, and Corollary \ref{Cordinary} becomes the Sylvester--Gallai condition: \emph{every line meeting $S^*$ in at least two points, other than the carriers of $S^*$ described below, contains a third point of $S^*$}. Since all boundary lines have finite positive slopes, we may work in the affine chart of the dual plane which writes the dual of $\{y=\sigma x+c\}$ as $(\sigma,c)$; the chart misses only the duals of vertical lines, of which there are none. In this chart $S^*$ is carried by finitely many vertical lines, one per slope -- the duals of the pencils of parallels; projectively the carriers are concurrent, at the dual point of the line at infinity, which lies outside the chart. For the carry set, $S^*$ consists of the arithmetic progressions $c\in\Z$ on the carriers $\sigma=1$ and $\sigma=\frac 12$, and $c\in\frac 12\Z$ on $\sigma=\frac 34$; one checks directly that the line through $(1,n)$ and $(\frac 12,m)$ meets the third carrier at $(\frac 34,\frac{n+m}2)$ -- the resonance in dual form -- so that every non-carrier connecting line meets $S^*$ in exactly three points. Figure \ref{fig:dual} shows this chart on the left and, on the right, a chart in which the common point of the carriers is finite and the whole configuration is bounded: there, as Proposition \ref{Pdual} shows, $S^*$ becomes \emph{exactly} the configuration of \cite{BDSYZ}.

	\begin{figure}[htb]
		\centering
		\includegraphics[width=.46\textwidth]{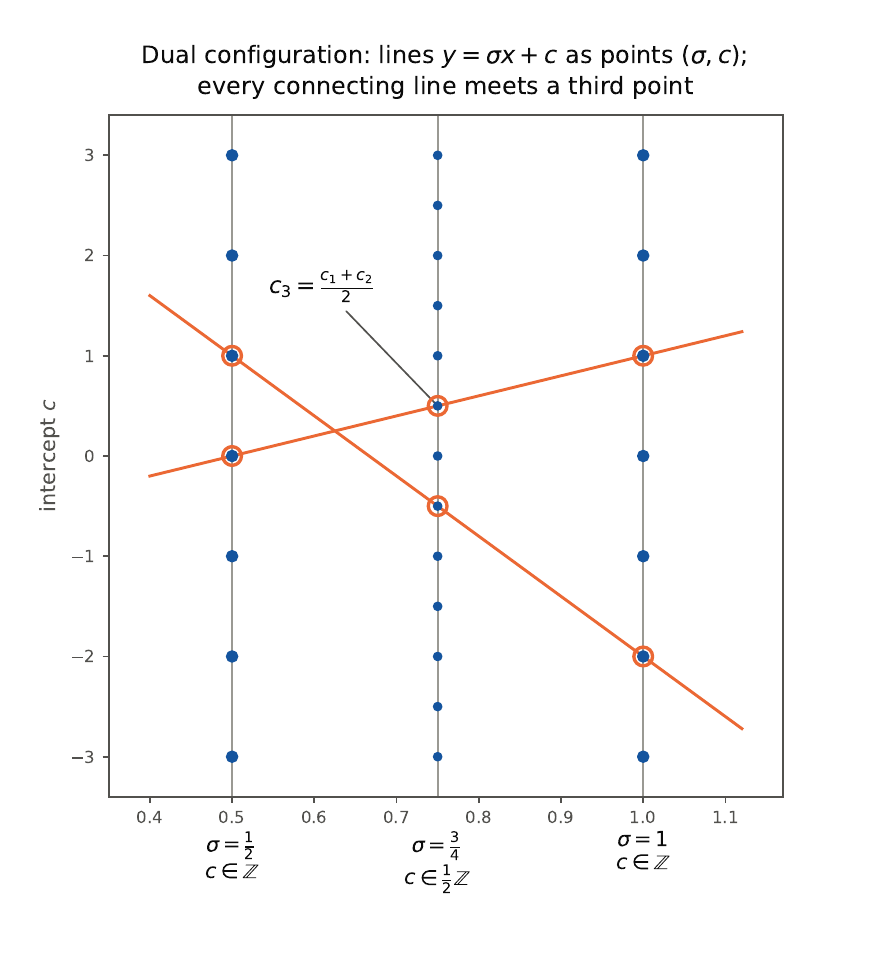}\hfill
		\includegraphics[width=.5\textwidth]{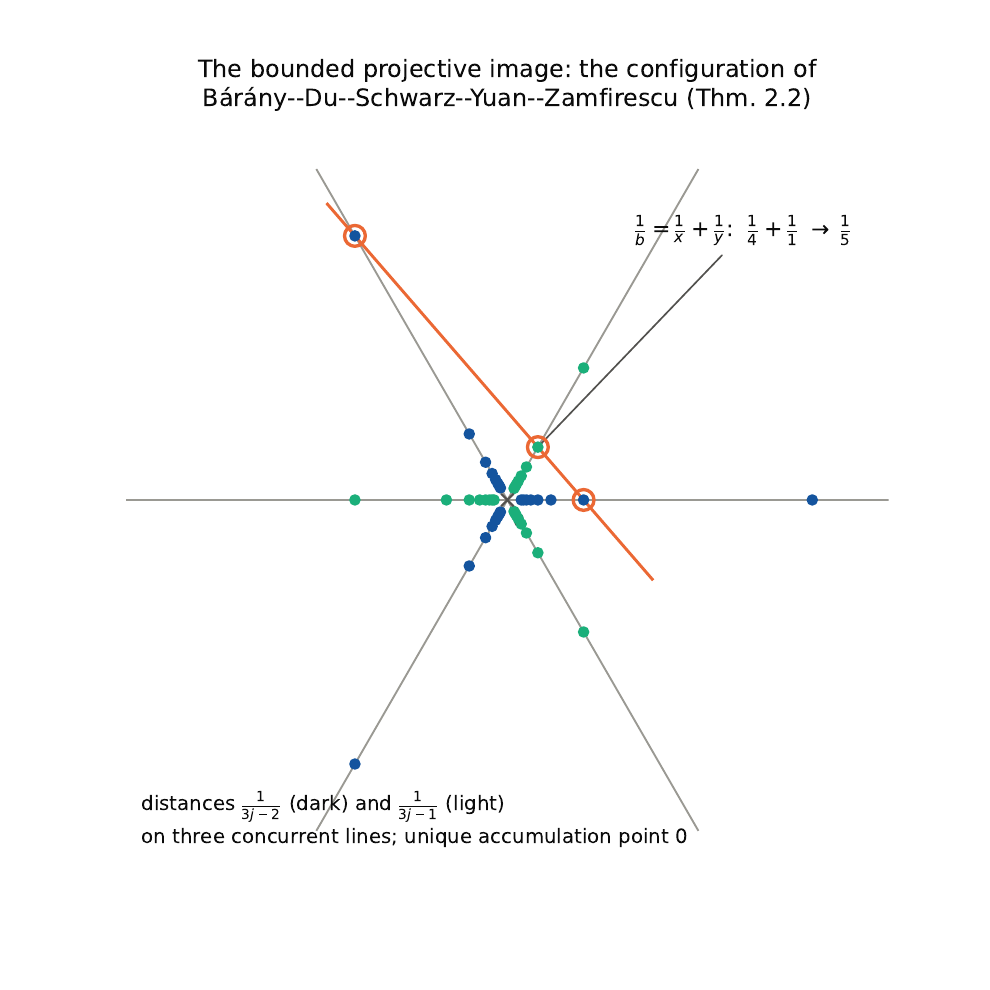}
		\caption{Left: the dual configuration $S^*$ in the affine chart $(\sigma,c)$ -- the boundary lines of the three-family set as points, with two connecting lines and the third-point law $c_3=(c_1+c_2)/2$; the three carriers meet at the dual of the line at infinity, outside the chart. Right: the same projective configuration in a chart where that common point is finite -- by Proposition \ref{Pdual}, exactly the configuration of \cite{BDSYZ}: six half-lines at angular spacing $\pi/3$ carrying the sequences of distances $1/(3j-2)$ and $1/(3j-1)$, with a collinear triple realizing $\frac1b=\frac1x+\frac1y$.}
		\label{fig:dual}
	\end{figure}

	\begin{remark}[collinear dual sets]
		The Sylvester--Gallai condition is empty when $S^*$ is collinear, and the collinear dual sets are exactly the duals of pencils -- the two degenerate invariant types. Points on a line through the dual of the line at infinity (a vertical line in the chart) are parallel boundary lines: strips. Points on any other line are concurrent boundary lines, the line being the dual of the common point: the fans of Proposition \ref{Pfan}, whose duals are the collinear sets of odd cardinality, and, in the limit of Theorem \ref{Tcone}, arbitrary closed subsets of the dual line of the vertex -- in the chart, sets with positive abscissae, reflecting the positive slopes. The carry configuration is the non-collinear case, to which the following proposition refers.
	\end{remark}

	\begin{proposition}\label{Pdual}
		In suitable coordinates, the projective dual of the arrangement
		$$\mathcal A'=\{x=m\}_{\,m\in 1+3\Z}\cup\{y=n\}_{\,n\in 1+3\Z}\cup\{x+y=k\}_{\,k\in 2+3\Z}$$
		-- the image of $\mathcal A_\triangle$ under $(x,y)\mapsto(3x+1,3y+1)$ -- is exactly the configuration of \cite{BDSYZ}, Theorem 2.2: six half-lines from a common point at angular spacing $\pi/3$, carrying the points at distances $1/(3j-2)$, $j\geq1$, from the vertex on three alternate half-lines and the points at distances $1/(3j-1)$ on the other three. Consequently $S^*$ is projectively equivalent to the configuration of \cite{BDSYZ}, which is a bounded realization of it.
	\end{proposition}

	\begin{proof}
		Dualize $ax+by+c=0\mapsto[a:b:c]$ and use the chart $[a:b:c]\mapsto(a/c,b/c)$, whose domain excludes only the duals of the lines through the origin; no line of $\mathcal A'$ passes through the origin, since $0$ lies in none of the three cosets. With $e_1,e_2$ the standard basis,
		$$\{x=m\}\mapsto-\tfrac1m\,e_1,\qquad \{y=n\}\mapsto-\tfrac1n\,e_2,\qquad \{x+y=k\}\mapsto-\tfrac1k\,(e_1+e_2),$$
		so the dual set is carried by the three lines spanned by $e_1$, $e_2$ and $e_1+e_2$: the three pencils of parallels pass through three points of the line at infinity, and their duals lie on three lines through the dual of the line at infinity, the origin of the chart. The coset $1+3\Z$ splits into $m=3j-2>0$ and $m=-(3j-1)<0$, $j\geq1$, giving on the first carrier the points at distances $1/(3j-2)$ along the ray of $-e_1$ and $1/(3j-1)$ along the ray of $e_1$; likewise on the second carrier, while the coset $2+3\Z$ splits into $k=3j-1$ and $k=-(3j-2)$, giving distances $1/(3j-1)$ along the ray of $-(e_1+e_2)$ and $1/(3j-2)$ along the ray of $e_1+e_2$. Now change coordinates by the linear map $L\colon e_1\mapsto(1,0)$, $e_2\mapsto(-\frac12,\frac{\sqrt3}2)$. Since $L(e_1+e_2)=(\frac12,\frac{\sqrt3}2)$ is again a unit vector, $L$ preserves all six sequences of distances while placing the six rays at directions $0^\circ,60^\circ,\dots,300^\circ$; the sequence $1/(3j-2)$ lands on the rays at $60^\circ$, $180^\circ$, $300^\circ$ and $1/(3j-1)$ on those at $0^\circ$, $120^\circ$, $240^\circ$. This is the set of \cite{BDSYZ}, Theorem 2.2. Under the identification, its Sylvester--Gallai property is the absence of ordinary vertices in $\mathcal A'$: a connecting line other than a carrier is the dual of a vertex of $\mathcal A'$ and carries one point of the configuration per line of $\mathcal A'$ through that vertex -- exactly three, by Theorem \ref{Tarr}. The bisector identity of \cite{BDSYZ}, Lemma 2.1 -- in an angle of $\frac{2\pi}3$ whose sides carry points at distances $x$ and $y$ from the vertex, the collinear point on the bisector lies at distance $b$ with $\frac1b=\frac1x+\frac1y$ -- is the resonance in dual form: it expresses the collinearity of the duals of $x=m$, $y=n$ and $x+y=m+n$, with $x=\frac1m$, $y=\frac1n$, $b=\frac1{m+n}$ for positive $m,n$ in the cosets. Finally, the boundary arrangement $u,v,u+v\in\Z$ of the carry set (Remark \ref{Rgeom}) is affinely equivalent to $\mathcal A_\triangle$, hence to $\mathcal A'$, and an affine map of the plane induces a projective transformation of the dual plane; therefore $S^*$ is projectively equivalent to the configuration of \cite{BDSYZ}.
	\end{proof}

	Dualizing Theorem \ref{Tarr} now shows that within its natural class the example of \cite{BDSYZ} is the only one.

	\begin{corollary}[dual form of Theorem \ref{Tarr}]\label{Cdual}
		Let $P$ be a countable non-collinear set in the real projective plane in which every line through two points of $P$ contains a third, and suppose that the dual line arrangement $\{p^*:\,p\in P\}$ is locally finite in some affine chart of the dual plane. Then, up to a projective transformation, $P$ is the configuration of Proposition \ref{Pdual}, possibly with the common point of the three carriers adjoined. (The augmented set still satisfies the hypotheses, since every line through the vertex and another point of the configuration is one of the carriers; so both cases occur, and there is no other example in this class.)
	\end{corollary}

	\begin{proof}
		At most one point of $P$ -- the dual of the chart's line at infinity, call it $z$ if it occurs -- fails to contribute an affine line in the chart; the duals of the remaining points form a locally finite arrangement $\mathcal A$ with at least two lines. $\mathcal A$ has no ordinary vertex: a vertex lying on exactly two of its lines would dualize to a line through exactly two points of $P$, and the guaranteed third point cannot be $z$, whose dual line, the line at infinity, misses the finite vertices of the chart. Theorem \ref{Tarr} applies. In its cases (i) and (ii) the lines of $\mathcal A$ are concurrent in the projective plane -- at a point at infinity of the chart or at a finite point -- so $P\setminus\{z\}$ lies on the dual line $\ell$ of the common point. Then either $P\subset\ell$, contradicting non-collinearity, or $z\notin\ell$ occurs in $P$, and the line through $z$ and any single point of $P\cap\ell$ contains no third point of $P$, contradicting the hypothesis. In case (iii), $\mathcal A$ is an affine image of $\mathcal A_\triangle$, so $P\setminus\{z\}$ is projectively equivalent to the dual of $\mathcal A_\triangle$, i.e.\ to the configuration of Proposition \ref{Pdual}; and if $z$ occurs, it is the dual of the line at infinity -- precisely the common point of the three carriers.
	\end{proof}

	\begin{remark}
		The Sylvester--Gallai theorem states that a \emph{finite} non-collinear point set always spans an ordinary line; dually, a finite arrangement of lines which are not all concurrent (in the projective sense, so in particular not all parallel) always has an ordinary point. Combined with Corollary \ref{Cordinary} this gives a purely combinatorial reason why a boundary arrangement with finitely many lines, neither all parallel nor all concurrent, is impossible; the concurrent case is exactly the fan, which Theorem \ref{Tlines} excludes through its hypothesis of finite cross-sections. Theorem \ref{Tarr} is the corresponding statement for infinite locally finite arrangements. Theorem 2.2 of \cite{BDSYZ} shows that Sylvester--Gallai fails for countable \emph{bounded} sets; by Proposition \ref{Pdual} their example is exactly the projective dual of an affine copy of the triangular lattice -- the configuration behind our positive-density example -- and by Corollary \ref{Cdual} it is, up to projective transformations and the adjoined vertex, the only counterexample whose dual arrangement is locally finite in some chart. The local finiteness cannot be dropped: $\mathbb Q^2$ is a countable set with no ordinary line, and its dual is a dense family of lines -- precisely the non-locally-finite case left open in Theorem \ref{Tlines}.
	\end{remark}

	\begin{remark}[the group law and the bound $D\leq 3$]
		Three concurrent carrier lines -- parallel in the chart of Figure \ref{fig:dual}, left -- form a degenerate cubic curve, and collinearity of one point on each carrier is an affine relation among the three intercepts -- the group law of the cubic, which is our resonance $\sum\mu_d\ell_d=0$. The elementary sumset argument at the end of the proof of Theorem \ref{Tarr} then forces the intercept sets to be arithmetic progressions with a common spacing. Since a cubic contains at most three lines, non-collinear configurations of Sylvester--Gallai type carried by parallel lines, each carrying infinitely many points, cannot involve more than three of them: this is the combinatorial counterpart of Proposition \ref{Pfour} and of the computational bound $D\leq 3$ of Remark \ref{Rcomp}. The collinear configurations -- the fans -- are exempt, and indeed use any odd number of carriers, one point on each.
	\end{remark}

	We record several questions suggested by this dictionary; a further reformulation, in terms of waves, is the subject of the final section.

	\begin{question}[quantitative stability]
		Sets defining few ordinary lines are classified by Green and Tao \cite{GT}. Is there an analytic counterpart: if $\|\H\chi_S-\chi_S\|_{L^2}\leq\delta\|\chi_S\|_{L^2}$ and $\partial S$ lies on countably many lines, is $S$ close in measure to a union of parallel strips, a fan, or a carry set? Theorem \ref{T1} is the model computation: truncating the strip creates ordinary points on the boundary, at the analytic cost $\sqrt{\varepsilon|\log\varepsilon|}$.
	\end{question}

	\begin{question}[Beatty--Fraenkel]
		Two-valued sums of floor functions are the subject of the theory of exact covers by Beatty sequences (Graham \cite{Gr}; Fraenkel's conjecture). Can the rigidity techniques of that theory prove Conjecture \ref{Lsturm} and settle the higher-rank cases of Proposition \ref{Pfour}?
	\end{question}

	\begin{question}[dyadic carries]
		The carry function is the $2$-cocycle of the extension $0\to\Z\to\R\to\R/\Z\to 0$; see \cite{DF} for its appearances in combinatorics. Base-$2$ carries are the cocycle of addition of dyadic integers, and the sibling identity of Section 3 is a carry identity in base $2$. Do the $\S$-invariant sets, modulo constants, consist exactly of dyadic strip patterns and a binary-carry configuration?
	\end{question}

	\begin{question}[elliptic carries]\label{Qell}
		The three-family example lives on a degenerate cubic. Is there an invariant set at positive density whose boundary lines envelope the dual curve of a smooth real cubic, with intercepts in a coset of the group and collinearity governed by the elliptic group law $P\oplus Q\oplus R=O$? None of the obstructions of Section \ref{sec:class}, which concern finitely many directions, apply to a continuum of directions.
	\end{question}

	\section{Two-valued superpositions of traveling waves}\label{sec:waves}

	Suppose, as in Section \ref{sec:class}, that the spectrum of $2\chi_S-1$ is contained in finitely many lines $L_1,\dots,L_n$ through the origin, and that the restriction to each line inverts to a bounded profile. Interpret the second coordinate as time, $t=y$. The wave along $L_d$ is then a function of $x-a_dt$, where the distinct speeds $a_d$ are the reciprocals of the slopes of the level lines, and
	\begin{equation}\label{wavesum}
		2\chi_S(x,t)-1=F(x,t)=\sum_{d=1}^{n}f_d(x-a_dt),
	\end{equation}
	a superposition of traveling waves; equivalently, $F$ is a bounded solution of the factored transport equation $\prod_{d=1}^n(\partial_t+a_d\partial_x)F=0$. The classification problem thus becomes a question about waves.

	\begin{question}\label{Qwaves}
		Let $f_1,\dots,f_n$ be nonconstant functions on the line and let $a_1,\dots,a_n$ be distinct constants. Can $F(x,t)=\sum_df_d(x-a_dt)$ take, for every $t$, only two values?
	\end{question}

	\begin{remark}[Galilean invariance]
		A boost $x\mapsto x-ct$ shifts every speed by $-c$ and preserves two-valuedness, so Question \ref{Qwaves} is invariant under changes of inertial frame, and the signs of the speeds are immaterial. The invariance $\H(2\chi_S-1)=2\chi_S-1$ corresponds to choosing a frame in which all speeds are positive -- a unidirectional field, with spectrum in $\overline{Q_{2,4}}$ -- and the annihilation of constants by $\H$ says that the equation does not see the time average of the field.
	\end{remark}

	\begin{theorem}\label{Twaves}
		Let $F=\sum_{d=1}^nf_d(x-a_dt)$ be a superposition of traveling waves with distinct speeds, as in \eqref{wavesum}, normalized so that $F$ is $\{0,1\}$-valued almost everywhere (the field $2\chi_S-1$ of \eqref{wavesum} is $\{\pm1\}$-valued; the two normalizations differ by an affine change $F\mapsto\frac12(F+1)$, which is absorbed into the profiles), with all profiles nonconstant.

		(i) $n=2$ is impossible; no boundedness of the profiles is required.

		(ii) $n=3$ occurs for every triple of distinct speeds $a_1<a_2<a_3$, in two forms. \emph{Sawtooth form}: for any $\lambda_1,\lambda_3>0$ with
		$a_2=(\lambda_1a_1+\lambda_3a_3)/(\lambda_1+\lambda_3)$ and any phases $\rho_1,\rho_3$, the field
		\begin{align*}
			F(x,t)=\{\lambda_1(x-a_1t)+\rho_1\}&+\{\lambda_3(x-a_3t)+\rho_3\}\\
			&-\{\lambda_1(x-a_1t)+\lambda_3(x-a_3t)+\rho_1+\rho_3\}
		\end{align*}
		is $\{0,1\}$-valued. \emph{Step form}: $F(x,t)=H(x-a_1t)+H(x-a_3t)-H(x-a_2t)$, together with its space-time translates $F(x-x_0,t-t_0)$ and its complement $1-F$, is $\{0,1\}$-valued -- this is the three-line fan of Proposition \ref{Pfan} in space-time; the three fronts must pass through a common event, since $H(u-a)+H(v-b)-H(u+v-c)$ is two-valued only for $c=a+b$. In both forms the inverted wave travels at the intermediate speed. Conversely, by Proposition \ref{Pchar} every solution with $n=3$ either has non-integer profiles, in which case the wavenumbers and frequencies satisfy the exact resonance $k_1+k_2=k_3$, $\omega_1+\omega_2=\omega_3$ and, granting Conjecture \ref{Lsturm}, the solution is of sawtooth form; or it has integer-valued profiles, where the step form occurs and we conjecture that it is the only one.

		(iii) Every odd $n$ occurs with step profiles: for speeds $a_1<\dots<a_n$ the alternating field of $n$ concurrent fronts
		$$F(x,t)=\frac12+\frac12\sum_{d=1}^n(-1)^{d-1}\tmop{sign}(x-a_dt)$$
		is $\{0,1\}$-valued (the fan with $k=n$ lines), and no even $n$ occurs with finitely many fronts. For profiles which are not integer-valued, $n=4$ is impossible when the four functionals satisfy a single resonance in general position, and we conjecture that $n\in\{1,3\}$ always; more precisely, that every solution with infinitely many fronts is of sawtooth form with $n=3$, and every solution with finitely many fronts is an alternating concurrent step field.
	\end{theorem}

	\begin{proof}
		(i) is Lemma \ref{Ltwo}: the essential-range argument uses no boundedness. (ii): the sawtooth field is two-valued by the carry identity, since the third argument is the sum of the first two, and the speed of the sum wave is the weighted mean, which fills $(a_1,a_3)$ as the weights vary. For the step field, at a time $t>0$ the three fronts sit at $a_1t<a_2t<a_3t$ and the values on the four complementary intervals are $0,1,0,1$; for $t<0$ the order of the fronts is reversed and the values are again $0,1,0,1$. This is Proposition \ref{Pfan}(iii) with $\ell_d=x-a_dt$; if $a_2$ were not the intermediate speed the value $2$ would appear. The converse statements are Proposition \ref{Pchar} together with Conjecture \ref{Lsturm}; in the resonant case $k_d=\lambda_d$ and $\omega_d=\lambda_da_d$, and positivity of the weights places the resonant speed strictly between the other two, in agreement with Lemma \ref{Lparity}. (iii) The alternating field is Proposition \ref{Pfan}(i) with $\ell_d=x-a_dt$: at any time $t\neq0$ the $n$ fronts are distinct and the field alternates $0,1,0,\dots,1$ across them. With finitely many fronts the arrangement is finite, and by Lemma \ref{Lparity} and the Sylvester--Gallai theorem it is either parallel ($n=1$) or concurrent with an odd number of lines. The impossibility statement for $n=4$ is Proposition \ref{Pfour}.
	\end{proof}

	\begin{remark}[collision rules]
		In the language of fronts, the jumps of the profiles are fronts moving at the $n$ speeds, and Lemma \ref{Lparity} becomes a collision rule: fronts of a two-valued field may only collide an odd number at a time, with orientations alternating in the order of the speeds; in a triple collision the middle-speed front is the inverted one. The alternating concurrent step fields are the elementary events -- an odd number of fronts meeting in a single collision -- and the sawtooth solution is a periodic train of triple collisions; the blow-up of Proposition \ref{Pfan}(iii) isolates one of them. The conditions $k_1+k_2=k_3$, $\omega_1+\omega_2=\omega_3$ are precisely the kinematic conditions of three-wave interaction familiar from nonlinear optics and plasma physics; here they are enforced arithmetically rather than dynamically. By Theorem \ref{Tarr}, the front arrangement of a two-valued field with locally finite fronts is a family of parallel fronts, a single collision event, or an affine image of the triangular space-time arrangement of the sawtooth solution.
	\end{remark}

	\begin{remark}[two physical incarnations]
		(a) The linearization of the one-dimensional Euler system about a constant state with velocity $u_0$ has characteristic speeds $u_0-c$, $u_0$, $u_0+c$: the entropy (contact) speed is the equal-weight mean of the acoustic speeds, which is the resonance of Theorem \ref{Twaves}(ii) with $\lambda_1=\lambda_3$. Consequently there are solutions of the linearized Euler equations whose density perturbation is, at every instant, the indicator of a set: two counter-propagating sawtooth acoustic trains interlocking with an entropy wave through the carry identity. The step solution is the linearized Riemann problem with an initial unit density step whose data are chosen so that the two acoustic fronts carry density jumps $+1$ and the entropy front carries $-1$. The pressure perturbation, which involves only the two acoustic modes, can never be two-valued, by (i).

		(b) For the vibrating string, (i) says that no genuine d'Alembert solution $f_1(x-ct)+f_2(x+ct)$ is two-valued at all times; the standing mode restores the possibility, since the speeds $-c,0,c$ are resonant with equal weights.
	\end{remark}

	\begin{remark}[self-similar fields]
		Theorem \ref{Tcone} describes all unidirectional two-valued fields which are homogeneous of degree $0$ in space-time, that is, of the form $F(x,t)=\Omega(\arg(x,t))$. In the frame where all speeds are positive: for $t>0$ the field equals a constant $\varepsilon$ on the half-line $x<0$ and an arbitrary $\{\pm1\}$-valued pattern $\omega(x/t)$ on $x>0$, with $\omega=\varepsilon$ near $0^+$ and $\omega=-\varepsilon$ near $+\infty$; for $t<0$ it is the point reflection of this picture. Such a field is a superposition of step waves over a continuum of speeds -- the linearized Riemann problem for a system with a continuum of characteristic speeds, such as free transport -- and it shows that Question \ref{Qwaves} loses all rigidity once the speeds are allowed to fill an interval: the two-valued self-similar fields are parametrized by arbitrary measurable subsets of the speed interval.
	\end{remark}

	\begin{proposition}[no singular profiles]\label{Pdist}
		Suppose that in \eqref{wavesum} the profiles are locally finite signed measures, while $F_t=F(\cdot,t)$ takes only two values for a.e.\ $t$. Then each $f_d$ is absolutely continuous, i.e. a locally integrable function.
	\end{proposition}

	\begin{proof}
		Let $s_d$ be the singular part of $f_d$, carried by a Borel set $N_d$ of Lebesgue measure zero. Since $F_t$ is a function, the singular parts must cancel: $\sum_ds_d(\cdot-a_dt)=0$ for a.e. $t$. Define signed measures on $\R^2$ by $M_d(A)=\iint\chi_A(x+a_dt,t)\,ds_d(x)\,dt$; testing against continuous functions and using Fubini gives $\sum_dM_d=0$. The measure $M_d$ is carried by $E_d=\{(x,t):x-a_dt\in N_d\}$. For $e\neq d$ pass to the coordinates $(u,v)=(x-a_dt,x-a_et)$, a nondegenerate linear change of variables since the speeds are distinct; in these coordinates $M_d$ is a constant multiple of $s_d\otimes m$, where $m$ is Lebesgue measure, while $E_e$ becomes $\R\times N_e$. Hence $|M_d|(E_e)=c\,(|s_d|\otimes m)(\R\times N_e)=0$: the measures $M_d$ are pairwise mutually singular. A vanishing finite sum of pairwise mutually singular measures has all its terms zero, so $s_d=0$ for every $d$.
	\end{proof}

	\begin{remark}
		Boundedness of the profiles can also be relaxed: for $n=2$ nothing is needed, and for $n=3$ the character step of Proposition \ref{Pchar} is unchanged, while the extraction of the linear part of an unbounded profile uses the Hyers--Ulam stability of Cauchy's functional equation (a measurable function with bounded Cauchy defect is the sum of a linear and a bounded function). For fully distributional profiles one may first apply $\prod_{d\neq e}(\partial_t+a_d\partial_x)$ to $F$, which isolates a derivative of the single profile $f_e$ and shows that each profile is, after finitely many integrations, a bounded function; together with Proposition \ref{Pdist} this reduces the general case to the one treated above. We omit the details.
	\end{remark}

	\begin{remark}[infinitely many speeds]
		For countably many speeds the finite mechanisms break down: the splitting of $e^{2\pi iF}$ into characters is a finite-product argument, and the torus-closure step of Proposition \ref{Pfour} requires finite rank. With step profiles, concurrent fans with countably many fronts (speeds accumulating at a limit speed, which then carries a renormalized step) are two-valued superpositions of infinitely many waves, and Theorem \ref{Tcone} goes further still. With sawtooth-type profiles we know of no two-valued superposition of infinitely many nonconstant waves that is not a disguised $n\leq 3$ example; the natural candidates (nested carries, products of carries) produce spectrum outside any prescribed finite family of lines, and Remark \ref{Rcomp} warns that cyclic models are unreliable here, since higher harmonics of a single direction alias across the cone.
	\end{remark}

	\begin{question}\label{Qcount}
		Is $n\in\{1,3\}$ also optimal among superpositions of countably many traveling waves with periodic, or almost periodic, profiles -- i.e.\ for atomic spectra supported on countably many lines through the origin?
	\end{question}

	\begin{question}\label{Qdisp}
		If the spectrum is supported on a curve through the origin, the field becomes a dispersive integral $F(x,t)=\int e^{i(kx-W(k)t)}\,d\sigma(k)$, with $\sigma$ a measure on the curve and $\omega=W(k)$ its dispersion relation. Can a genuinely dispersive wave field be two-valued at all times? For piecewise-linear dispersion relations (finitely many phase speeds) the answer is given by Theorem \ref{Twaves}, and for spectra filling a cone by Theorem \ref{Tcone}; the case of a curved dispersion relation is the wave form of Question \ref{Qell}.
	\end{question}

	The wave formulation places the classification next to several classical subjects. \emph{Tomography}: by the Fourier slice theorem, spectrum on $L_1,\dots,L_n$ means that $\chi_S$ retains no information beyond its $n$ X-ray transforms, and the profiles $f_d$ are that projection data. Fishburn, Lagarias, Reeds and Shepp characterized the sets uniquely determined by finitely many projections by additivity: $S=\{\sum_dg_d(\ell_d)>0\}$, see \cite{FLRS1, FLRS2}. Our condition is the exact equality $\chi_S=\sum_dg_d(\ell_d)$, and the ``switching components'' (ghosts) of that theory are precisely the alternating configurations excluded by Lemma \ref{Lparity} and by the transition analysis in the proof of Theorem \ref{Tlines}. \emph{Ridge functions}: sums $\sum_dg_d(\ell_d)$ are sums of ridge functions in the sense of approximation theory \cite{Pi}; the representation problem for indicators with exact values $\{0,1\}$ does not seem to have been considered there. \emph{Periodic decompositions}: decompositions of integer-valued functions into integer-valued periodic summands with prescribed periods were studied by K\'arolyi, Keleti, K\'os and Ruzsa \cite{KKKR}, see also the survey \cite{FR}; Question \ref{Qwaves} is a directional version, with prescribed speeds in place of prescribed periods, and Conjecture \ref{Lsturm} belongs to the same circle of ideas as the theory of exact covers by Beatty sequences \cite{Gr}. \emph{Crystalline measures}: for fixed $t$ the derivative $\partial_xF_t$ is a signed atomic measure whose support translates at $n$ speeds while its spectrum remains on fixed lines -- a moving relative of the crystalline measures and Fourier quasicrystals of \cite{LO, KS, AKKV}; the integer-valued counting functions arising there from Lee--Yang polynomials are themselves integer-valued sums of waves, and it is plausible that this machinery can prove Conjecture \ref{Lsturm}. We have not found Question \ref{Qwaves}, as stated, in the literature.

	\bigskip
	\footnotesize
	\noindent\textsc{Authors' addresses}

	\medskip
	\noindent Evgeny Abakumov\\
	Univ Gustave Eiffel, Univ Paris Est Cr\'eteil, CNRS, LAMA UMR 8050,\\
	F-77447 Marne-la-Vall\'ee, France\\
	\textit{E-mail:} \texttt{evgueni.abakoumov@univ-eiffel.fr}

	\medskip
	\noindent Komla Domelevo\\
	Department of Mathematics, University of W\"urzburg, W\"urzburg, Germany

	\medskip
	\noindent Stefanie Petermichl\\
	Department of Mathematics, University of W\"urzburg, W\"urzburg, Germany

	\medskip
	\noindent A. Poltoratski\\
	Department of Mathematics, University of Wisconsin,\\
	Van Vleck Hall, 480 Lincoln Drive, Madison, WI 53706, USA\\
	\textit{E-mail:} \texttt{poltoratski@wisc.edu}

\end{document}